\documentclass[runningheads]{llncs}
\usepackage{graphicx}
\usepackage[symbol]{footmisc}

\usepackage{float}
\usepackage{mathtools}
\usepackage{amssymb, amsmath, latexsym}
\usepackage{nicefrac}
\usepackage{hhline}
\usepackage{makecell}
\usepackage{caption}
\usepackage{multirow}

\usepackage{colortbl}
\definecolor{bgcolor}{rgb}{0.8,1,1}
\definecolor{bgcolor2}{rgb}{0.8,1,0.8}
\usepackage{threeparttable}
\newcommand{\myred}[1]{{\color{red}#1}}

\usepackage{pifont}
\usepackage[colorinlistoftodos,bordercolor=orange,backgroundcolor=orange!20,linecolor=orange,textsize=scriptsize]{todonotes}
\usepackage{algorithm}\usepackage{algpseudocode}

\allowdisplaybreaks
\graphicspath{ {images/} }
\def \R {\mathbb R}

\newtheorem{assumption}{Assumption}
  
\def\R{\mathbb{R}}

\def\R{\mathbb R}

\def\EE{\mathbb E}

\newcommand{\eqdef}{\vcentcolon=}

\newcommand{\cO}{\mathcal{O}}

\DeclareMathOperator{\dom}{\mathrm{dom}}

\newcommand{\Lavg}{\overline{L}}

\def\<#1,#2>{\langle #1,#2\rangle}

\begin{document} 
\title{Optimal Analysis of Method with Batching for Monotone Stochastic Finite-Sum Variational Inequalities}
\titlerunning{Optimal Method with Batching for Monotone Stochastic VIs}
%
\author{Alexander Pichugin\inst{1}\and
Maksim Pechin\inst{1}\and
Aleksandr Beznosikov\inst{1,2} \and
Alexander Gasnikov\inst{1,2}
}
\authorrunning{A. Pichugin, M. Pechin, A. Beznosikov, A. Gasnikov}
%
\institute{Moscow Institute of Physics and Technology, Moscow, Russia, \and
Institute for Information Transmission Problems, Moscow, Russia
}
\maketitle              
\begin{abstract}
Variational inequalities are a universal optimization paradigm that is interesting in itself, but also incorporates classical minimization and saddle point problems. Modern realities encourage to consider stochastic formulations of optimization problems. In this paper, we present an analysis of a method that gives optimal convergence estimates for monotone stochastic finite-sum variational inequalities. In contrast to the previous works, our method supports batching and does not lose the oracle complexity optimality. The effectiveness of the algorithm, especially in the case of small but not single batches is confirmed experimentally.
\keywords{stochastic optimization \and variational inequalities \and finite-sum problems \and batching}
\end{abstract}
\section{Introduction}
In this paper, we consider the following variational inequality problem: 
\begin{equation}
    \label{eq:VI}
    \text{Find } x^* \in \mathcal{X} \text{ such that } \langle F(x^*), x - x^* \rangle + g(x) - g(x^*) \geq 0, \text{ for all } x \in \mathcal{X},
\end{equation}
where $F$ is some operator, $g$ is a proper convex lower semicontinuous function with domain $\text{dom} g$. Variational inequalities are a popular optimization formulation. This is due to the fact that they incorporate many different problem statements that arise widely in various fields of applied science. To understand the importance of variational inequalities, let us consider two important examples.  

\begin{example}[Minimization]
Consider the classic and widely known convex regularized minimization problem:
\begin{align}
\label{eq:min}
\min_{x \in \R^d} f(x) + g(x),
\end{align}
where $f$ is typically a main target function, and $g$ a regularizer. 
If we put $F(x) := \nabla f(x)$, then one can prove that $x^* \in \dom g$ is a solution for \eqref{eq:VI} if and only if $x^* \in \dom g$ is a solution for \eqref{eq:min} \cite[Section 1.4.1]{VIbook2003}.
\end{example}

\begin{example}[Saddle point]
Consider the convex-concave saddle point problem:
\begin{align}
\label{eq:minmax}
\min_{x_1 \in \R^{d_1}} \max_{x_2 \in \R^{d_2}} f(x_1,x_2) + g_1 (x_1) - g_2(x_2),
\end{align}
where $g_1$ and $g_2$ can also be interpreted as regularizers. 
If we define $F(x) := F(x_1,x_2) := [\nabla_{x_1} f(x_1,x_2),$ $ -\nabla_{x_2} f(x_1,x_2)]$ and $g(x) = g(x_1,x_2) = g_1 (x_1) + g_2(x_2)$, then it can be proved that $x^* \in \dom g$ is a solution for \eqref{eq:VI} if and only if $z^* \in \dom g$ is a solution for \eqref{eq:minmax} \cite[Section 1.4.1]{VIbook2003}. It gives that convex-concave saddle point problems  \eqref{eq:minmax} can be investigated via a reformulation in the form of the variational inequality \eqref{eq:VI}.
\end{example}

While the minimization problems are usually considered separately from the paradigm of variational inequalities, the saddle point problems are very often studied in the generality of variational inequalities. But even excluding minimization problems, variational inequalities themselves \cite{scutari2010convex} and their special case of saddle problems have plenty of real-world applications. First of all, it is important to note the classic examples of economics and game theory, which had their golden age in the middle of the last century \cite{NeumannGameTheory1944,HarkerVIsurvey1990,VIbook2003}. It is also important to note also the classic but newer stories from robust optimization \cite{BenTal2009:book,nesterov2007dual,nemirovski2004prox} and supervised/unsupervised learning \cite{Thorsten,bach2011optimization,NIPS2004_64036755,bach2008convex,esser2010general,chambolle2011first}. And the cherry on top of the enduring popularity of variational inequalities are all-new applications from reinforcement learning \cite{Omidshafiei2017:rl,Jin2020:mdp}, adversarial training \cite{Madry2017:adv}, and generative models \cite{goodfellow2014generative,daskalakis2017training,gidel2018variational,mertikopoulos2018optimistic}. 

Meanwhile, every year modern applied problems become more and more complicated and very often have a stochastic nature: $F(x) = \mathbb{E}_{\xi \sim \mathcal{D}}[F_{\xi}(x)]$.  Often such mathematical expectation cannot be represented in a closed form due to the fact that the distribution $\mathcal{D}$ is non-trivial or even unknown, therefore one often resorts to the Monte Carlo approach, which approximates the expectation integral using a set of samples from the distribution $\mathcal{D}$:
\begin{equation}
    \label{eq:sum}
    F(x) = \frac{1}{M} \sum\limits_{m=1}^M F_m (x).
\end{equation}
In algorithms designed to solve stochastic problems \cite{robbins1951stochastic,johnson2013accelerating,defazio2014saga,nguyen2017sarah}, one often avoids calling the full operators/gradients, and resorts to the batching technique (e.g., we can choose one or more terms from \eqref{eq:sum}). Therefore, it is very important that the algorithm is robust to the use of any size batches, both in terms of theory and in terms of practice. But even for minimization problems, not all stochastic methods satisfy such properties \cite[Section 5]{allen2017katyusha}. This brings us to the objective of this paper:

\begin{quote}
\textit{
Develop an algorithm for solving the monotone stochastic (finite-sum) variational inequality with Lipschitz summands \eqref{eq:VI}+\eqref{eq:sum}. In terms of theory, the algorithm should be non-sensitive to the size of batches.
}
\end{quote}

\textbf{Our contribution and related works.} The theory of solving monotone variational inequalities with Lipschitz operators has an extensive history. A natural idea to get a method is to use a gradient descent with replacement of the gradient $\nabla f(x)$ by the operator $F(x)$. But such a method for monotone operators may diverge. A breakthrough was the creation of the ExtraGradient method \cite{korpelevich1976extragradient}, which was widely theoretically investigated \cite{mokhtari2020unified,nemirovski2004prox} and subsequently received various modifications/analogues \cite{popov1980modification,tseng2000modified}, including the stochastic case \cite{juditsky2011solving}. Regarding the methods for monotone stochastic finite-sum problems, one can highlight papers \cite{carmon2019variance,alacaoglu2021stochastic,Yura2021} where the authors adapt the variance reduction technique to variational inequalities.  Moreover, the results from \cite{alacaoglu2021stochastic} are optimal \cite{han2021lower}. But this is valid only for the case when each iteration we call $1$ term from the sum \eqref{eq:sum}. As soon as we use a non-single-batch, the results from works \cite{carmon2019variance,alacaoglu2021stochastic} become worse. See Table \ref{tab:comparison2} for details.

We solve this issue and propose a method that is optimal for any batch size. Our method is based on the optimistic scheme with momentum from \cite{kovalev2022optimalvi}, but in this paper, the authors studied only strongly monotone variational inequalities, while we consider more general and more complicated from the point of view of theoretical analysis monotone ones. 

\renewcommand{\arraystretch}{2}
\begin{table*}[h]
    \centering
    \small
	\caption{Summary complexities for finding an $\varepsilon$-solution for the monotone stochastic (finite-sum) variational inequality with Lipschitz terms \eqref{eq:VI}+\eqref{eq:sum}. Convergence is measured by the gap function. {\em Notation:} $\mu$ = constant of strong monotonicity of the operator $F$, $L$ = Lipschitz constants for all $F_{i}$, $n$ = the size of the local dataset, $b$ = the batch size per iteration }
    \label{tab:comparison2}    
  \begin{threeparttable}
    \begin{tabular}{|c|c|c|}
    \hline
    \textbf{Reference } & \textbf{Complexity} & \textbf{Additional assumptions} 
    \\\hline
    Nemirovski et al. \cite{nemirovski2004prox}\tnote{{\color{blue}(1)}} & $\mathcal{O} \left( \myred{n} \frac{L}{\varepsilon} \right)$  &\\\hline
    Carmon et al. \cite{carmon2019variance} & $\mathcal{\tilde O} \left( \sqrt{\myred{b}n}\frac{L}{\varepsilon} \right)$ & \makecell{{$x \to \langle F(x) + \nabla g(x), x - u \rangle$ is convex}  \\ {for any $u$ \textbf{or} bounded domain}} \\\hline
    Alacaoglu and Malitsky \cite{alacaoglu2021stochastic} & $\mathcal{O} \left( \sqrt{\myred{b}n}\frac{L}{\varepsilon}\right)$ & \\\hline
    Alacaoglu et al. \cite{Yura2021} & $\mathcal{O} \left( \myred{n}\frac{L}{\varepsilon} \right)$ & 
    \\\hline
    \cellcolor{bgcolor2}{This paper} & \cellcolor{bgcolor2}{$\mathcal{O} \left( \sqrt{n}\frac{L}{\varepsilon} \right)$} & \cellcolor{bgcolor2}{} \\\hline \hline
    Han et al. \cite{han2021lower} & $\Omega \left( \sqrt{n}\frac{L}{\varepsilon}\right)$ & lower bounds
    \\\hline
    \end{tabular}   
    \begin{tablenotes}
    \scriptsize
    \item [{\color{blue}(1)}] deterministic methods, similar results were also obtained in \cite{tseng2000modified,mokhtari2020unified}
\end{tablenotes}    
    \end{threeparttable}
\end{table*}

\section{Problem Setup and Assumptions} 

We consider the problem \eqref{eq:VI}+\eqref{eq:sum}, where $\mathcal{X}$ be a finite dimensional vector space with Euclidean inner product $\langle \cdot, \cdot \rangle$ and induced norm $\| \cdot \|$. For a proper convex lower semicontinuous $g$, we denote domain as $\text{dom} g = \{x: g(x) < + \infty \}$. We also assume that the function $g$ is proximally friendly, i.e. the following operator $\text{prox}_{\alpha g} (x) = \arg\min_{y \in \mathcal{X}} \left\{ \alpha g(y) + \frac{1}{2} \| y - x\|^2\right\}$ with any $\alpha > 0$ can be exactly calculated for free.
For $y = \text{prox}_{\alpha g}(x)$ it holds that 
\begin{align}
x-y\in \partial (\alpha g)(y)\label{prox_property}. 
\end{align}
\\By $\mathbb{E}_{k}$ we denote a conditional expectation 
$\mathbb{E}\left[\cdot | S^{k-1},S^{k-2},\cdot\cdot\cdot,S^{0}\right].$
\\Each operator $F_i: \dom g \to \R^d$ is single-valued. The key assumptions of our paper we give as follows
\begin{assumption}\label{ass}
\hspace{0.cm}
\begin{itemize}
    \item The solution (may be not unique) for the problem \eqref{eq:VI}+\eqref{eq:sum} exists.
    \item The operator $F$ is monotone, i.e. for all $u, v \in \R^d$ we have
\begin{equation*}
\langle F(u) - F(v); u - v \rangle \geq 0.
\end{equation*}
    \item The operator $F$ is $L$-Lipschitz, i.e. for all $u, v \in \R^d$ we have
\begin{equation*}
\| F(u) - F(v) \| \leq L \| x - y \|.
\end{equation*}
\item Each operator $F_m$ is $L_m$-Lipschitz. We define $\bar L$ as $\bar L^2 = \frac{1}{M}\sum\limits_{m=1}^M L_m^2$.
\end{itemize}
\end{assumption}

\section{Main Part}

We base our method on the non-distributed version of Algorithm 1 from \cite{kovalev2022optimalvi}. This algorithm is based on the so-called optimistic modification \cite{popov1980modification} of the ExtraGradient method \cite{korpelevich1976extragradient}. The essence of this modification is to use the value of the operator not only at the points of the current iteration ($x^k$ and $w^k$), but also from the past iteration ($x^{k-1}$ and $w^{k-1}$). We also use the variance reduction technique \cite{johnson2013accelerating,alacaoglu2021stochastic}, for this we introduce an additional sequence $w^k$. The point $w^k$ (the reference point) changes rarely (if $p$ is small), and hence the full operator $F(w^k)$ is almost not recalculated. We also apply the random negative momentum $\gamma (w^k - x^k)$, which is necessary to use the variance reduction approach in methods for variational inequalities \cite{alacaoglu2021stochastic,Yura2021}. 

\begin{algorithm}[h]
	\caption{Optimistic Method with Momentum and Batching}
	\label{alg:sarah}
	\begin{algorithmic}[1]
\State
\noindent {\bf Parameters:}  stepsize $\eta>0$, momentum $\gamma > 0$, probability $p \in (0;1)$, batchsize $b \in \{1,\ldots,M\}$, number of iterations $K$\\
\noindent {\bf Initialization:} choose  $x^0 = w^0 = x^{-1} = w^{-1} \in \R^d$ \label{Alg1Line2}
\For {$k=0, 1, \ldots, $ }
    \State Sample $j_1^k,\ldots, j_b^k$ independently from $\{1,...,M\}$ uniformly at random
    \State $S^k=\{j_1^k,..., j_b^k\}$
    \State $\Delta^k = \frac{1}{b}\sum\limits_{j\in S^k}(F_j(x^k)-F_j(w^{k-1})+(F_j(x^k)-F_j(x^{k-1})))+F(w^{k-1})$ \label{Alg1Line6}
    \State $x^{k+1}=\text{prox}_{\eta g}(x^k+\gamma (w^k-x^k)-\eta \Delta^k)$ \label{Alg1Line7}
    \State $w^{k+1}=\begin{cases}
    x^{k+1},\; \text{with probability } p\\
    w^k,\; \text{with probability } 1-p
\end{cases}$\label{Alg1Line8}
\EndFor
	\end{algorithmic}
\end{algorithm}

We use the gap function as convergence criterion:
\begin{equation}
    \label{gap}
    \text{Gap} (z) \eqdef \sup_{u \in \mathcal{C}} \left[ \langle F(u),  z - u  \rangle + g(x) - g(u)\right].
\end{equation}
Here we do not take the maximum over the entire set $\mathcal{X}$, but over $\mathcal{C}$ -- a compact subset of $\mathcal{X}$. Thus, we can also consider unbounded sets $\mathcal{X} \subseteq \R^d$. This is permissible, since such a version of the criterion is valid if the solution $x^{*}$ lies in $\mathcal{C}$; for details see \cite{nesterov2007dual}.
The following theorem gives the convergence of the proposed method. 

\begin{theorem}\label{th:ALg1}
Consider the problem \eqref{eq:VI}+\eqref{eq:sum} under Assumptions~\ref{ass}. Let  $\{x^k\}$ be the sequence generated by Algorithm~\ref{alg:sarah} with tuning of $\eta, \theta, \alpha, \beta, \gamma$  as follows:
$$ 0 < p=\gamma \leq \frac{1}{16}, \quad \eta = \min\left\{\frac{\sqrt{\gamma b}}{8\bar L}, \frac{1}{8L}\right\}$$
Then, given $\varepsilon>0$, the number of iterations for $\EE[\text{Gap}(x^k)] \leq \varepsilon$ is 
\begin{equation*}
     \cO\left( \frac{1}{\sqrt{pb}}\frac{\bar L}{\varepsilon}+  \frac{L}{\varepsilon} \right).
\end{equation*}
\end{theorem}
See the proof in Appendix \ref{app_proof1}. Theorem gives iteration complexity of Algorithm \ref{alg:sarah}, but it is primarily the oracle complexity (the number of calls of terms $F_i$) that is of interest. Note that Theorem gives iterationionic complexity of Algorithm 1, but it is primarily the main point of interest is the oracle complexity (the number of calls to terms $F_i$). One can note that at each iteration we call $\mathcal{O}\left(b + pn\right)$ terms in average/expectation  -- each time we call the batch with size of $b$ and with probability $p$ we compute the full operator. From where we immediately get the optimal choice for $p$:
\begin{corollary}
Under the conditions of Theorem \ref{th:ALg1}, if we choose $p = \frac{b}{n}$, then we get the following oracle  complexities
\begin{equation*}
     \cO\left( \sqrt{n}\frac{\bar L}{\varepsilon}+  b \frac{L}{\varepsilon} \right).
\end{equation*}
\end{corollary}
With $b \leq \frac{\Lavg \sqrt{n} }{L}$, we get $\cO\left( \sqrt{n}\frac{\bar L}{\varepsilon}\right)$. And this results is optimal and unimproveable \cite{han2021lower}.

\section{Experiments}

In this section, we aim to test the performance of Algorithm~\ref{alg:sarah} in practice. 

We consider the bilinear problem:
\begin{align}
    \label{bilinear}
     \textstyle  \min\limits_{x\in \triangle^d} \max\limits_{y\in \triangle^d}  x^\top A y,
\end{align}
where $\triangle^d$ is the unit simplex in $\R^d$. We use the same experimental setup as in \cite{alacaoglu2021stochastic}, in particular we consider the
policeman and burglar matrix from \cite{nemirovski2013mini} and the first test matrix from \cite{nemirovski2009robust}. 
Note that the problem \eqref{bilinear} does not have the finite-sum form as \eqref{eq:VI}+\eqref{eq:sum}, but we can rewrite $A$ from \eqref{bilinear} as follows $A = \sum_{i=1}^d A_{i:}$ or $A = \sum_{i=1}^d A_{:i}$, where $A_{i:}$ is the $i$th row of $A$ and $A_{:i}$ is the $i$th column of $A$ -- see details in Section 5.1.2 from \cite{alacaoglu2021stochastic}.

For comparison, we take methods from Table \ref{tab:comparison2}. In particular, we choose Algorithms 1 and 2 from \cite{alacaoglu2021stochastic}, Algorithm 1+2 from \cite{carmon2019variance}.
All algorithms are considered in the Euclidean setting with the projection to simplex from \cite{condat2016fast}.

The parameters of all methods are tuned for the best convergence among the theoretical possible -- see Section 6 from \cite{alacaoglu2021stochastic}. We run all methods with different batch sizes. We use
duality gap \eqref{gap} as the convergence measure, it can be simply computed as $\left[\max_i (A^T x)_i - \min_j (Ay)_j\right]$ for simplex constraints. The comparison criterion is the number of operations (one operation is computationally equal to calculations of $Ay$ and $A^T x$). 

The results are reflected in Figures \ref{fig:comp1} and \ref{fig:comp2}. They show that Algorithm \ref{alg:sarah} outperforms all competitors. 

\begin{figure}[h!]
\centering
\caption{Comparison of computational complexities for Algorithm~\ref{alg:sarah}, \texttt{EG-Mal22-1} (Algorithm 1 \cite{alacaoglu2021stochastic}), \texttt{EG-Mal22-2} (Algorithm 2 \cite{alacaoglu2021stochastic}), \texttt{EG-Car19} (Algorithm 1+2 from \cite{carmon2019variance}) on \eqref{bilinear} with policeman and burglar matrix from \cite{nemirovski2013mini}.}
\label{fig:comp1}
\begin{minipage}[][][b]{\textwidth}
\centering
\includegraphics[width=0.35\textwidth]{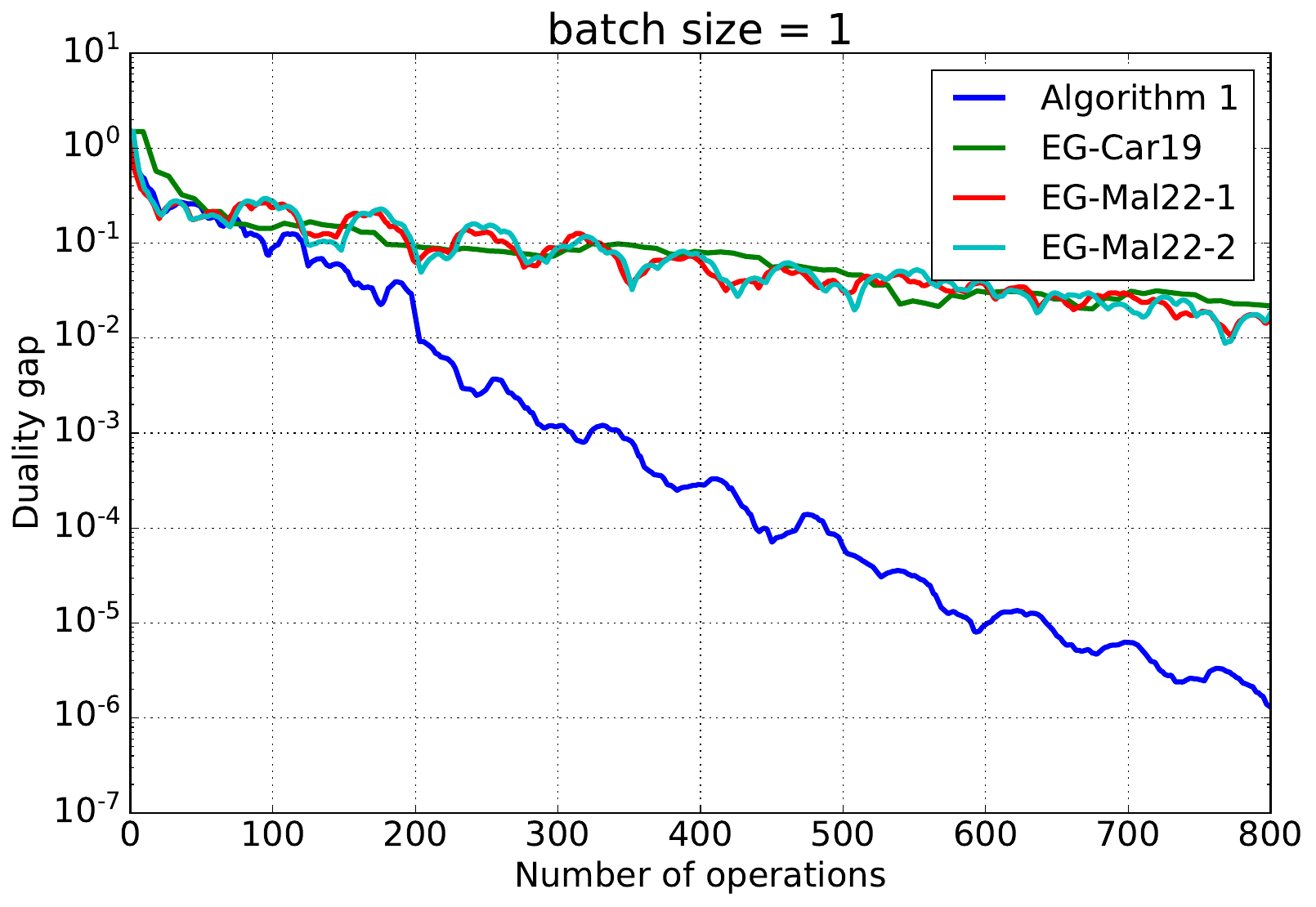}
\includegraphics[width=0.35\textwidth]{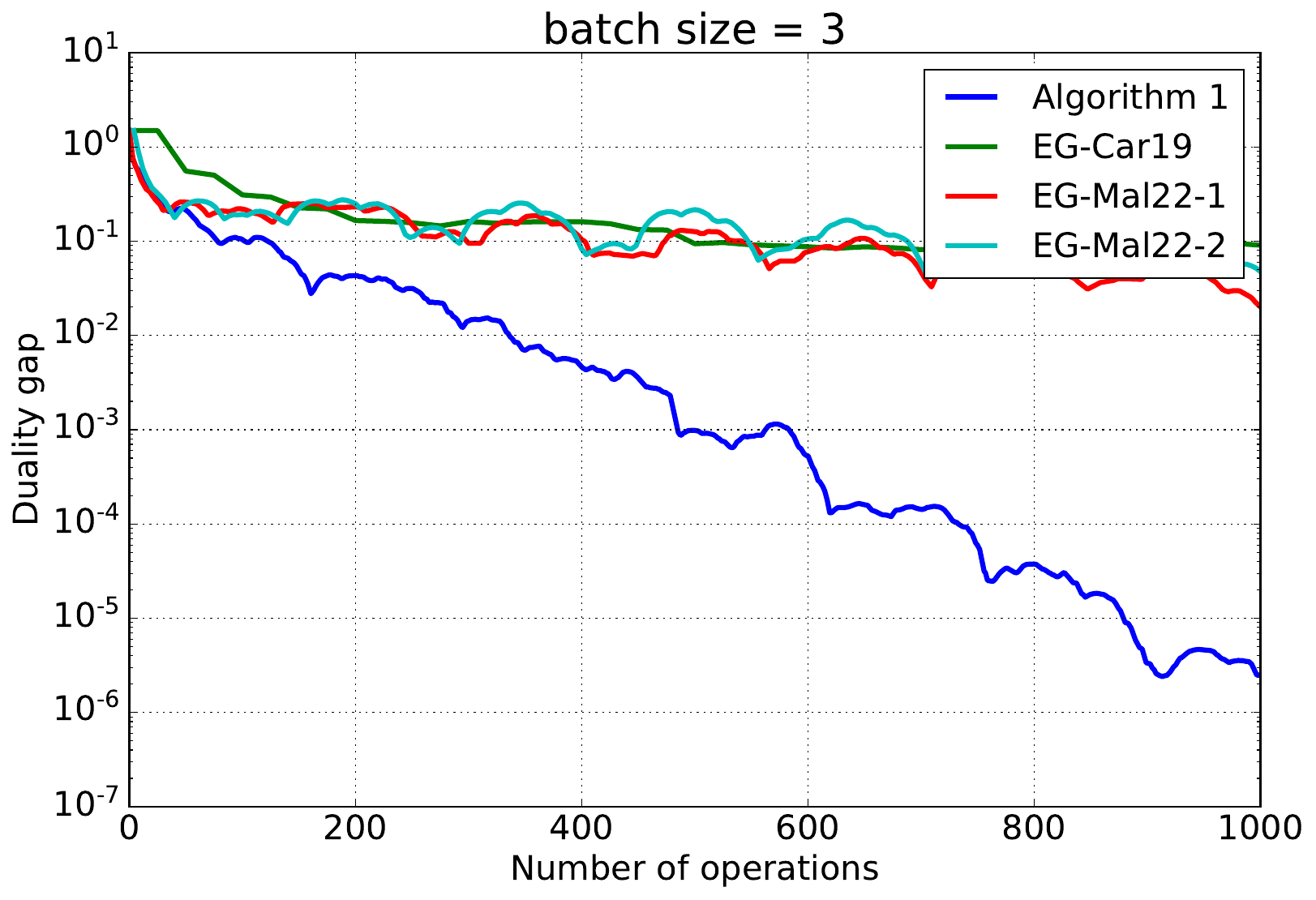}
\includegraphics[width=0.35\textwidth]{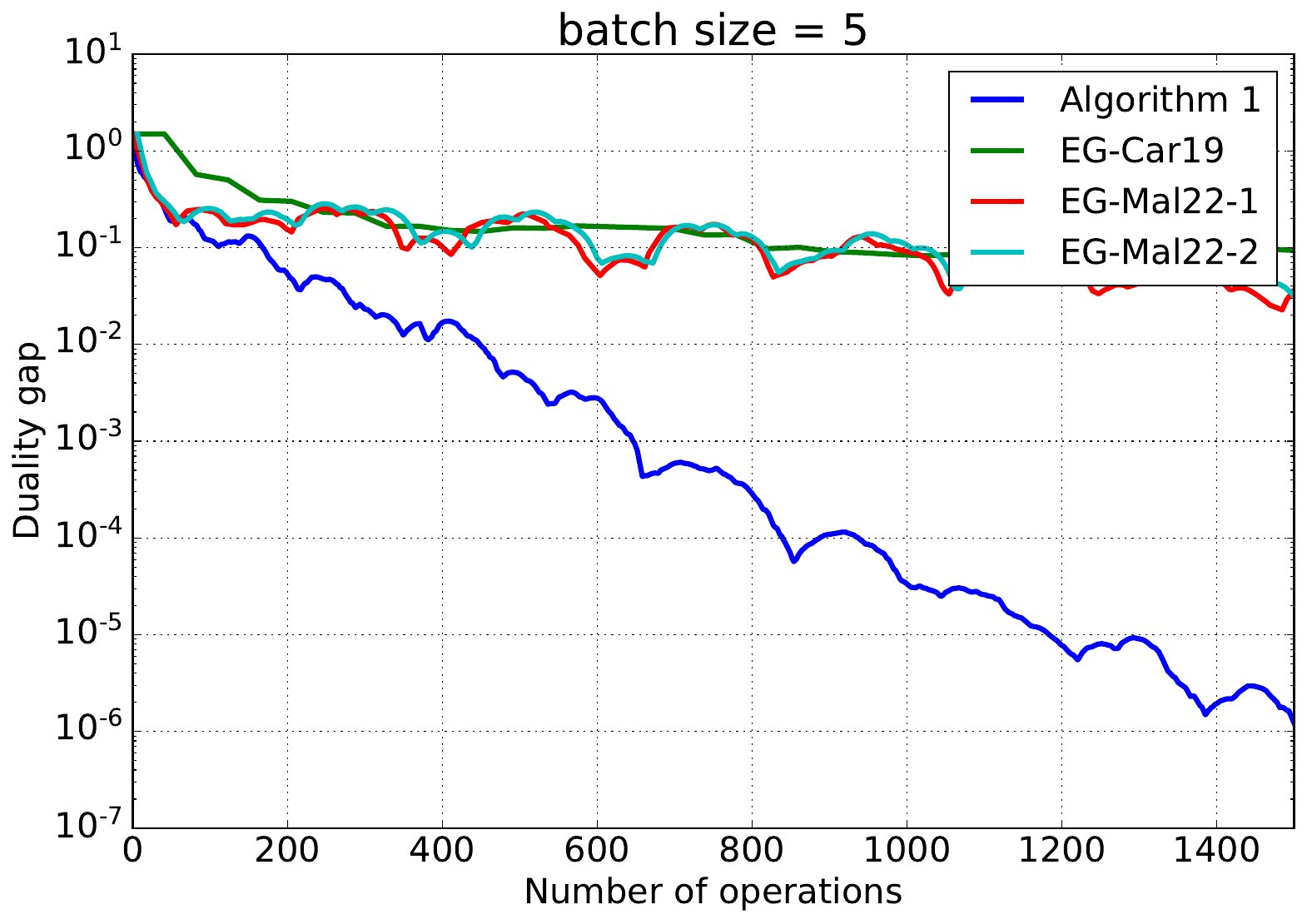}
\includegraphics[width=0.35\textwidth]{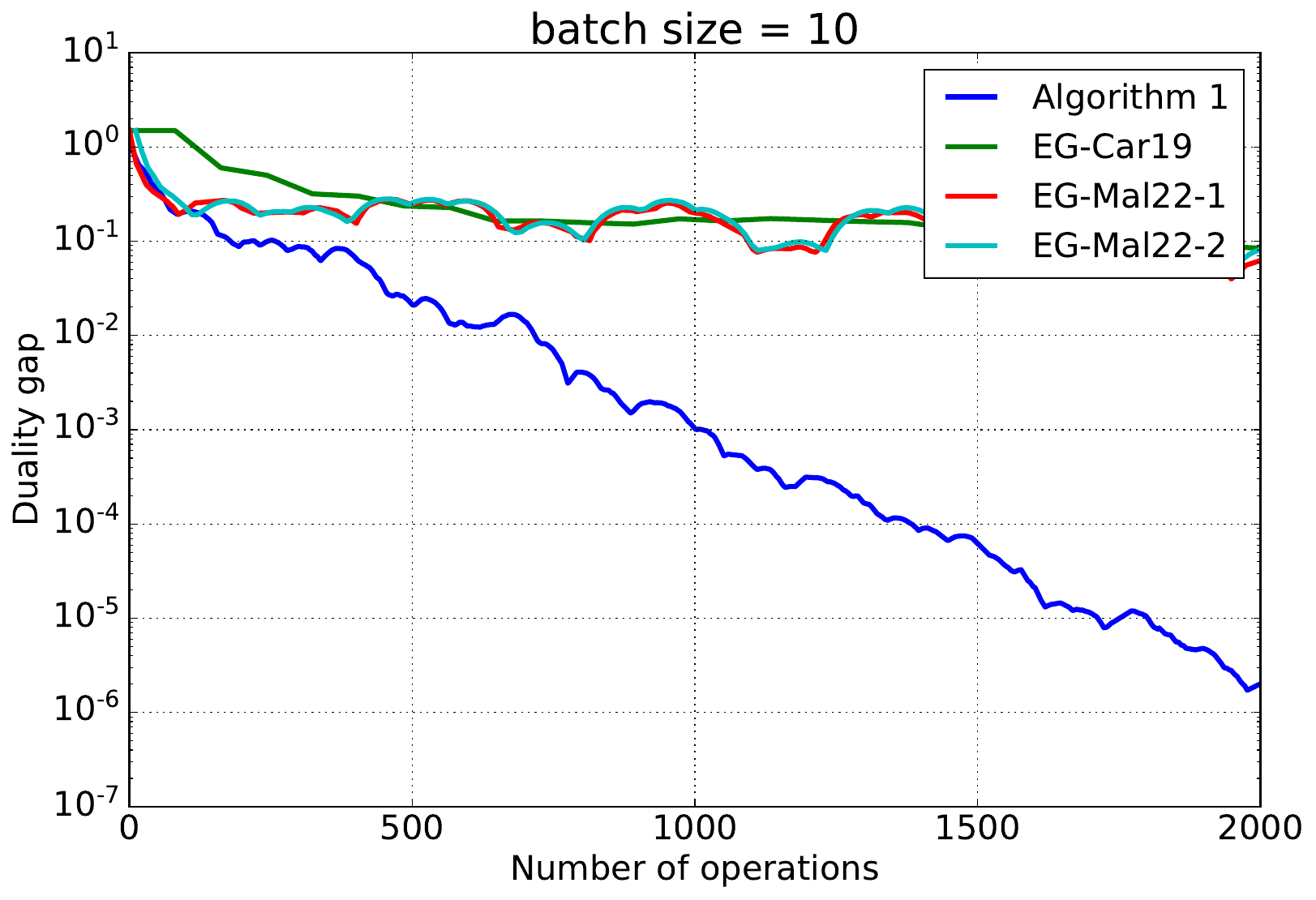}
\end{minipage}
\end{figure}

\begin{figure}[h!]
\centering
\caption{Comparison of computational complexities for Algorithm~\ref{alg:sarah}, \texttt{EG-Mal22-1} (Algorithm 1 \cite{alacaoglu2021stochastic}), \texttt{EG-Mal22-2} (Algorithm 2 \cite{alacaoglu2021stochastic}), \texttt{EG-Car19} (Algorithm 1+2 from \cite{carmon2019variance}) on \eqref{bilinear} with policeman and burglar matrix from \cite{nemirovski2013mini}.}
\label{fig:comp2}
\begin{minipage}[][][b]{\textwidth}
\centering
\includegraphics[width=0.35\textwidth]{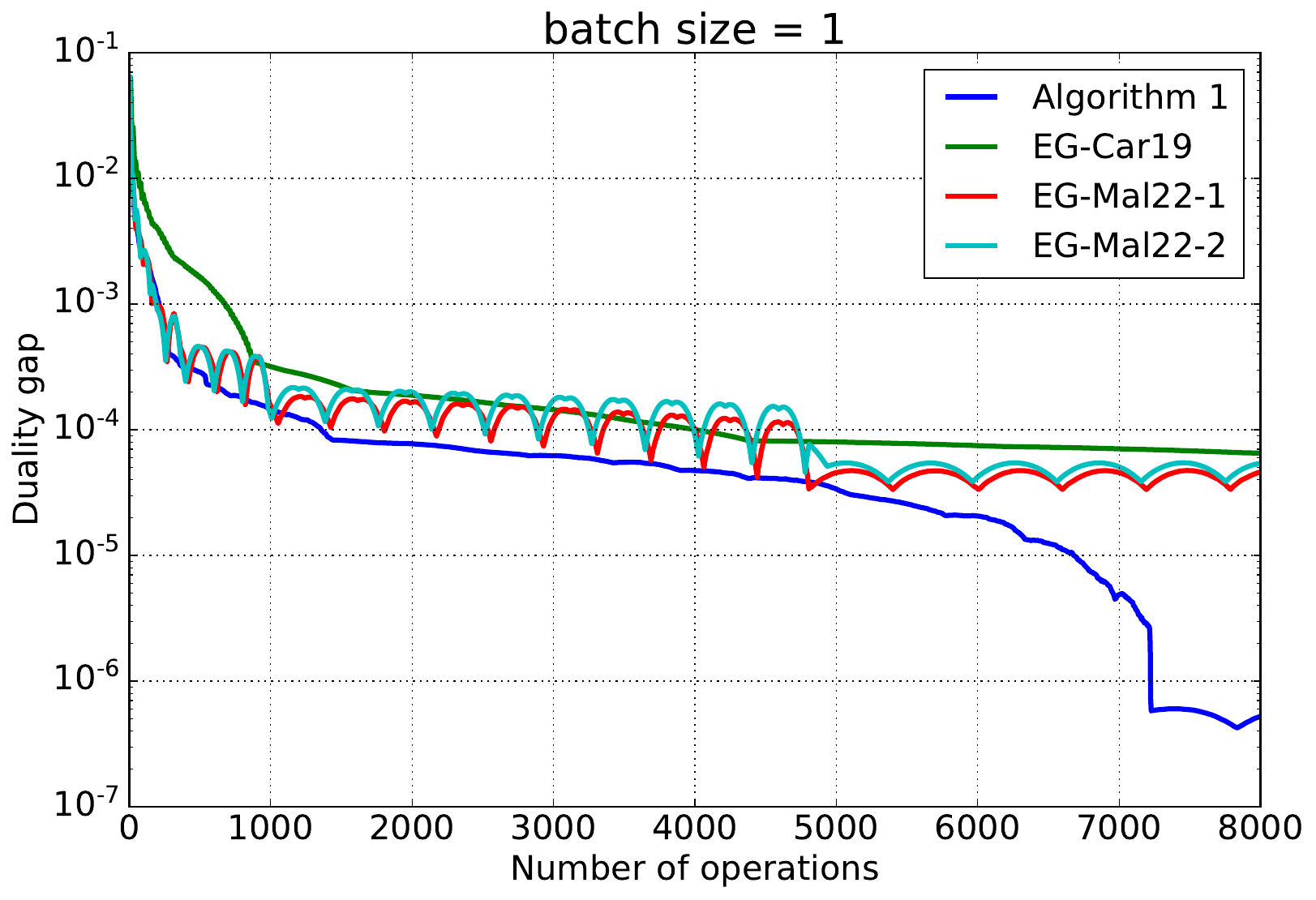}
\includegraphics[width=0.35\textwidth]{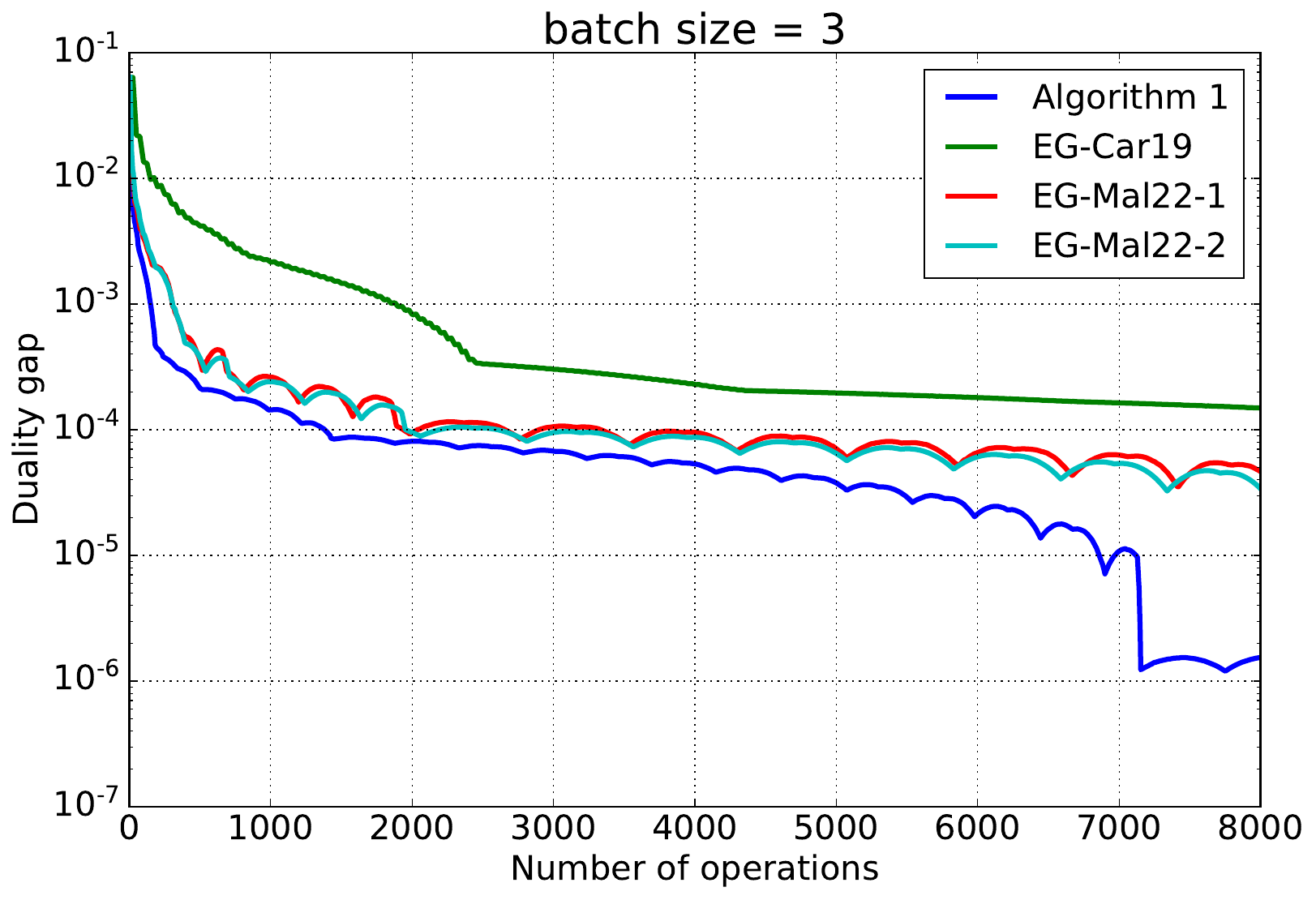}
\includegraphics[width=0.35\textwidth]{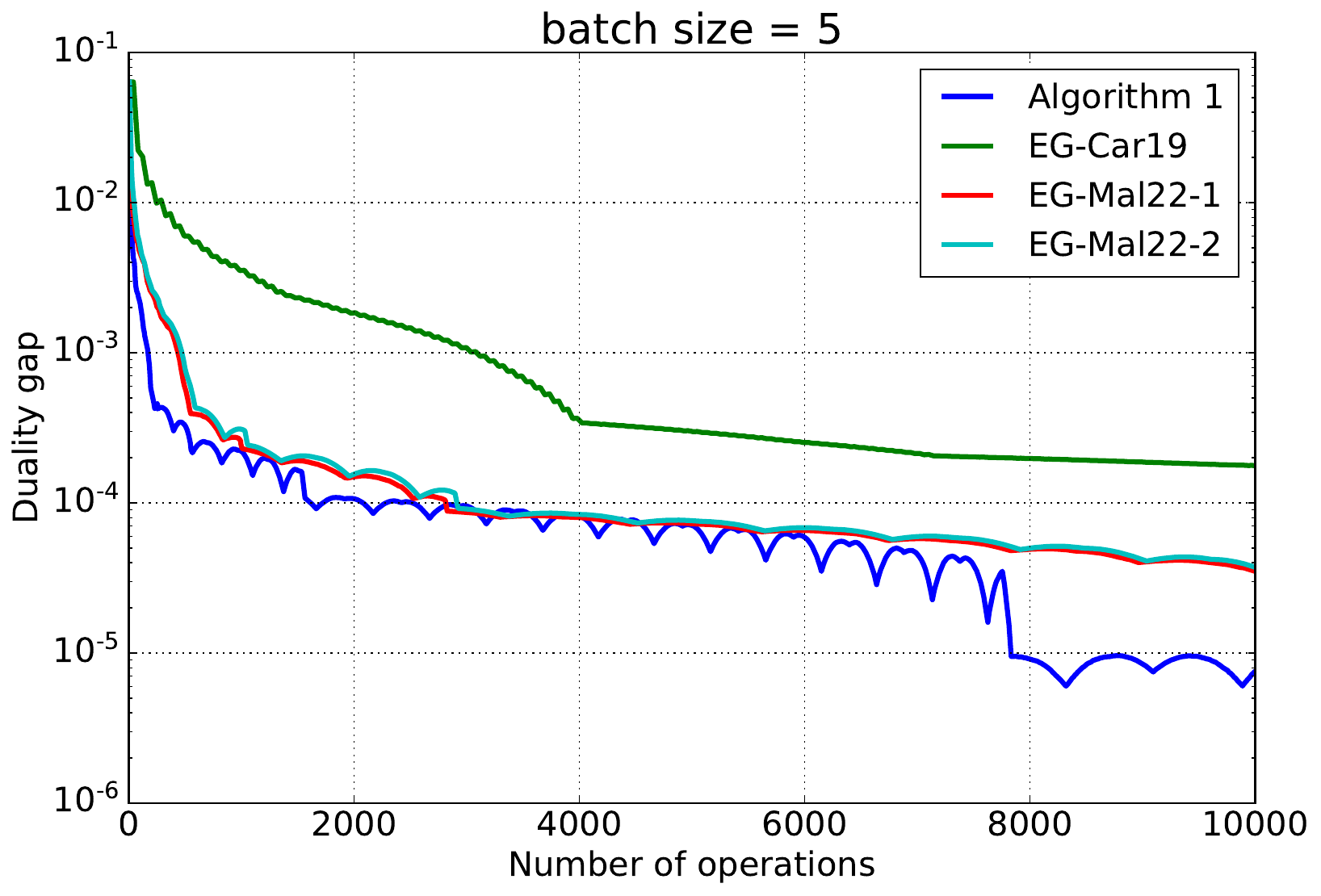}
\includegraphics[width=0.35\textwidth]{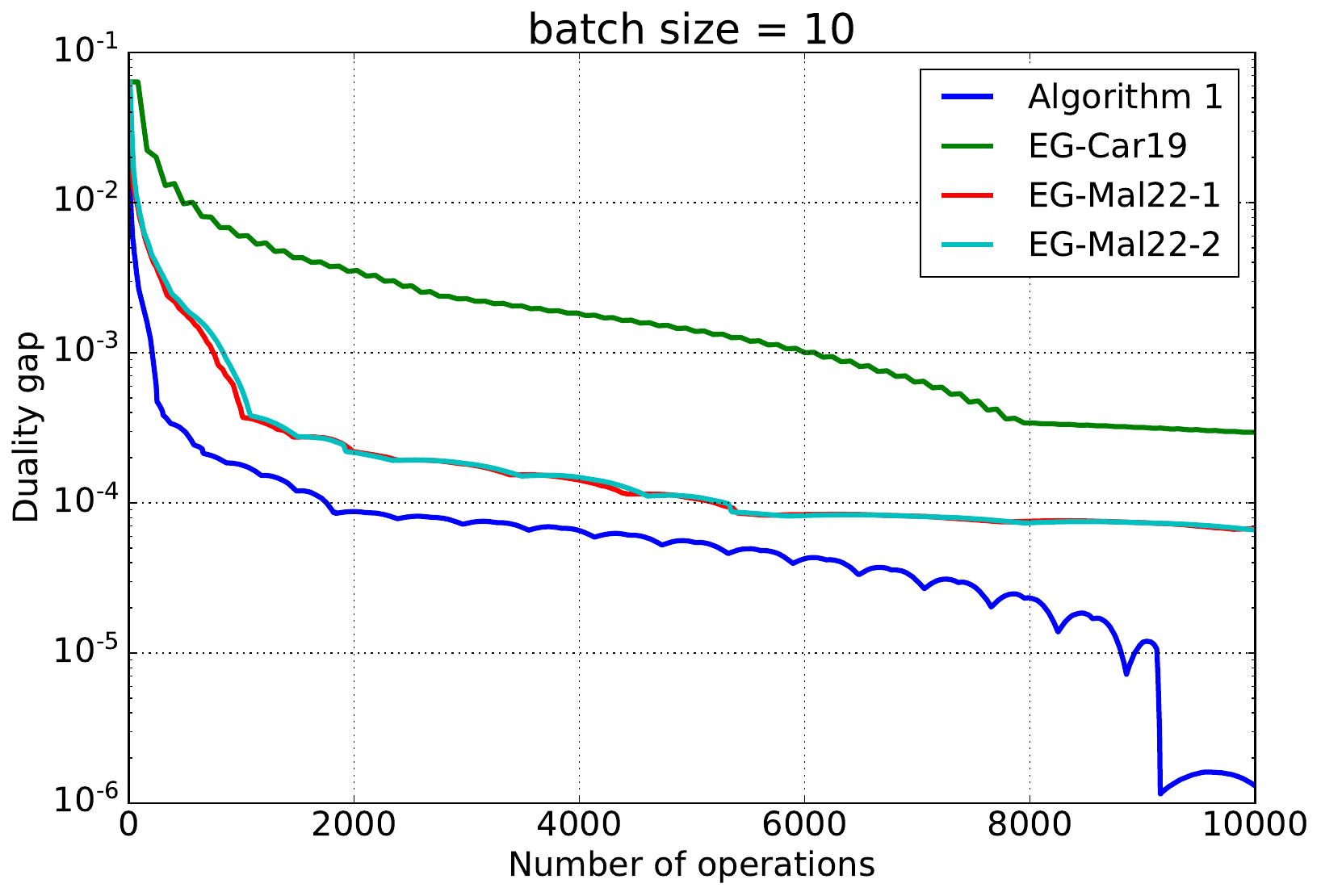}
\end{minipage}
\end{figure}

\section{Conclusion}

We presented an algorithm for solving monotone stochastic (finite-sum) variational inequalities with Lipschitz terms. From the theoretical perspective, the algorithm is optimal for almost any batches size. As directions for future work, it would be interesting to get a version of Algorithm \ref{alg:sarah} in the non-Euclidean case with arbitrary Bregman divergence.

\section*{Acknowledgements}

The work of A. Pichugin and M. Pechin  was supported by a grant for research centers in the field of artificial intelligence, provided by the Analytical Center for the Government of the Russian Federation in accordance with the subsidy agreement (agreement identifier 000000D730321P5Q0002) and the agreement with the Moscow Institute of Physics and Technology dated November 1, 2021 No. 70-2021-00138.

\bibliographystyle{splncs04}
\bibliography{ltr}

\newpage 

\appendix

\section{Proof of Theorem \ref{th:ALg1}}\label{app_proof1}
Before proving the theorem we introduce the following Lemmas.
\begin{lemma}\label{lem:1}[Lemma 2.4 from \cite{alacaoglu2021stochastic}]
    Let $\mathcal{F} = (\mathcal{F}_k)_{k\geq 0}$ be a filtration and $(u^k)$ a stochastic process adapted to $\mathcal{F}$ with $\mathbb{E}[u^{k+1}|F_k]=0$. Then for any $K\in \mathbb{N}$, $x^0\in X$ and any compact set $\mathcal{C}\subseteq X$
    \begin{align*}
        \mathbb{E}\left[\max\limits_{x\in \mathcal{C}}\sum\limits_{k=0}^{K-1}\langle u^{k+1}, x \rangle\right] \leq \max\limits_{x\in \mathcal{C}}\frac{1}{2}\|x^0-x\|^2+\frac{1}{2}\sum\limits_{k=0}^{K-1}\mathbb{E}\|u^{k+1}\|^2.
    \end{align*}
\end{lemma}

\begin{lemma} \label{lem:3} Under Assumption~\ref{ass} for iterates of Algorithm \ref{th:ALg1} the following inequality holds: 
\begin{equation}
    \mathbb{E}\left[\left\|\Delta^{k}-\mathbb{E}_{k}\left[\Delta^{k}\right]\right\|^{2}\right]
    \leq{\frac{2{\overline{{L}}}^{2}}{b}}\mathbb{E}\left[\left\|x^{k}-w^{k-1}\right\|^{2}
    +\left\|x^{k}-x^{k-1}\right\|^{2}\right],
\end{equation}
where $\mathbb{E}_k\left[\Delta^k\right]$ is equal to 
\begin{equation}
    \mathbb{E}_{k}\left[\Delta^{k}\right]=2F(x^{k})-F(x^{k-1}).
\end{equation}
\end{lemma}
\begin{proof}
We start from line \ref{Alg1Line6} of Algorithm \ref{alg:sarah}
\begin{align*}
\mathbb{E}_k\Big[\big\|\Delta^k & -\mathbb{E}_k\left[\Delta^k\right]\big\|^2\Big] 
\\ 
 = &
 \mathbb{E}_k\Bigg[\bigg\|\frac{1}{b}\sum\limits_{j\in S^k}(F_j(x^k)-F_j(w^{k-1})+(F_j(x^k)-F_j(x^{k-1})))+F(w^{k-1})
 \\& - (2 F(x^k) - F(x^{k-1}))\bigg\|^2\Bigg].
\end{align*}
With Cauchy–Schwarz inequality, we have
 \begin{align*}
 \mathbb{E}_k\Big[\big\|\Delta^k & -\mathbb{E}_k\left[\Delta^k\right]\big\|^2\Big]
 \\
 \leq &
 2 \mathbb{E}_k\left[\left\|\frac{1}{b} \sum_{j \in S^k}\left(F_j\left(x^k\right)-F_j\left(w^{k-1}\right)\right)-\left(F\left(x^k\right)-F\left(w^{k-1}\right)\right)\right\|^2\right] \notag
 \\
& +2 \mathbb{E}_k\left[\left\|\frac{1}{b} \sum_{j \in S^k}\left(F_j\left(x^k\right)-F_j\left(x^{k-1}\right)\right)-
\left(F\left(x^k\right)-F\left(x^{k-1}\right)\right)\right\|^2\right].
\end{align*}
Using we choose $j_1^k,\ldots, j_b^k$ in $S^k$ indepdently and uniformly, one can note that
 \begin{align*}
&\mathbb{E}_k \Big[\langle \left( F_j\left(x^k\right)-F_j\left(w^{k-1}\right)\right)-\left(F\left(x^k\right)-F\left(w^{k-1}\right) \right), 
\\&\hspace{2cm}\left( F_j\left(x^k\right)-F_j\left(w^{k-1}\right)\right)-\left(F\left(x^k\right)-F\left(w^{k-1}\right) \right)\rangle \Big]
\\
&\hspace{1cm}=\mathbb{E}_k \bigg[\langle \mathbb{E}_{j^k_i} \left[\left( F_{j^k_i}\left(x^k\right)-F_{j^k_i}\left(w^{k-1}\right)\right)-\left(F\left(x^k\right)-F\left(w^{k-1}\right) \right) \right], 
\\
&\hspace{2cm}\mathbb{E}_{j^k_l} \left[\left( F_{j^k_l}\left(x^k\right)-F_{j^k_l}\left(w^{k-1}\right)\right)-\left(F\left(x^k\right)-F\left(w^{k-1}\right) \right) \right]\rangle \bigg]
\\
&\hspace{1cm}=0.
\end{align*}
Hence, we get
 \begin{align*}
 \mathbb{E}_k\Big[\big\|\Delta^k & -\mathbb{E}_k\left[\Delta^k\right]\big\|^2\Big]
 \\
\leq&2 \mathbb{E}_k\left[\sum_{j \in S^k}\frac{1}{b^2}\left\|\left(F_j\left(x^k\right)-F_j\left(w^{k-1}\right)\right)-\left(F\left(x^k\right)-F\left(w^{k-1}\right)\right)\right\|^2\right] \notag
  \\
&+ 2 \mathbb{E}_k\left[\sum_{j \in S^k}\frac{1}{b^2}\left\|\left(F_j\left(x^k\right)-F_j\left(x^{k-1}\right)\right)-\left(F\left(x^k\right)-F\left(x^{k-1}\right)\right)\right\|^2\right] \notag
\\
=&\frac{2}{b^2} \mathbb{E}_k\left[\sum_{j \in S^k}\left\|\left(F_j\left(x^k\right)-F_j\left(w^{k-1}\right)\right)-\left(F\left(x^k\right)-F\left(w^{k-1}\right)\right)\right\|^2\right] \notag
  \\
&+\frac{2}{b^2} \mathbb{E}_k\left[\sum_{j \in S^k}\left\|\left(F_j\left(x^k\right)-F_j\left(x^{k-1}\right)\right)-\left(F\left(x^k\right)-F\left(x^{k-1}\right)\right)\right\|^2\right] \notag
\\
 \leq & \frac{2}{b^2} \mathbb{E}_k\left[\sum_{j \in S^k}\left\|F_j\left(x^k\right)-F_j\left(w^{k-1}\right)\right\|^2\right]
\\
&+\frac{2}{b^2} \mathbb{E}_k\left[\sum_{j \in S^k}\left\|F_j\left(x^k\right)-F_j\left(x^{k-1}\right)\right\|^2\right]. \notag
\end{align*}
In the last step, we used the fact that $\mathbb{E}\|X-\mathbb{E}X\|^2 = \mathbb{E}\|X\|^2-\|\mathbb{E}X\|^2$. Next, we again take into account that $j_1^k,\ldots, j_b^k$ in $S^k$ are chosen uniformly
 \begin{align*}
 \mathbb{E}_k\Big[\big\|\Delta^k & -\mathbb{E}_k\left[\Delta^k\right]\big\|^2\Big]
 \\
 \leq &\frac{2}{b} \mathbb{E}_k\Bigg[\mathbb{E}_{j \sim \text { u.a.r. }\{1, \ldots, M\}}\bigg[\left\|F_j\left(x^k\right)-F_j\left(w^{k-1}\right)\right\|^2 
 \\& +\left\|F_j\left(x^k\right)-F_j\left(x^{k-1}\right)\right\|^2\bigg]\Bigg] \notag
 \\
 = &\frac{2}{M b} \sum_{j=1}^M \left(\left\|F_j\left(x^k\right)-F_j\left(w^{k-1}\right)\right\|^2+\left\|F_j\left(x^k\right)-F_j\left(x^{k-1}\right)\right\|^2\right) .
\end{align*}
Since each operator $F_j$ is $L_j$-Lipschitz, we can rewrite it as
\begin{equation*}
    \mathbb{E}_{k}\left[\left||\Delta^{k}-\mathbb{E}_{k}\left[\Delta^{k}\right]|\right|^{2}\right]\leq{\frac{2}{m b}}\sum_{j=1}^{m}L_{j}^{2}\left(||x^{k}-w^{k-1}||^{2}+\,\|x^{k}-x^{k-1}\|^{2}\right).
\end{equation*}
Applying the definition of $\overline{L}$, we obtain 
\begin{equation*}
    \mathbb{E}_{k}\left[\|\Delta^{k}-\mathbb{E}_{k}\left[\Delta^{k}\right]\|^{2}\right]\leq{\frac{2{\overline{{L}}}^{2}}{b}}\left(\|x^{k}-w^{k-1}\|^{2}+\left\|x^{k}-x^{k-1}\right\|^{2}\right).
\end{equation*}
Taking the full expectation concludes the proof. \hfill\(\Box\)
\end{proof}
\begin{lemma}\label{lem:4}
For iterates of Algorithm \ref{th:ALg1} with $\gamma = p$ the following bound holds for any compact set $\mathcal{C}\subseteq X$: 
\begin{equation*}
    \mathbb{E}\left[\max\limits_{x\in \mathcal{C}}\sum\limits_{k = 0}^{K - 1}e_1(x, k)\right] \leq 2\max\limits_{x \in \mathcal{C}} \{\|x - x^0\|^2\} + \frac{\gamma(1 - \gamma)}{2}\sum\limits_{k = 0}^{K - 1}\mathbb{E}\|x^{k+1} - \omega^k\|^2,
\end{equation*}
where $e_1(k,x) 
    = 
    \|w^{k+1}-x\|^2-\|w^k-x\|^2+(1-\gamma) \|x^{k+1}-x\|^2$.
\end{lemma}
\begin{proof}
For shortness we introduce
\begin{align*}
    u^{k+1} = \gamma x^{k+1} + (1 - \gamma)\omega^k - \omega^{k+1}.
\end{align*}
With new notation, we can rewrite $e_1(k,x)$ as
\begin{align*}
    e_1(k,x) = 2\left\langle u^{k+1}, x\right\rangle - \gamma \|x^{k+1}\|^2-(1-\gamma)\|w^k\|^2+\|w^k\|^2.
\end{align*}
From line \ref{Alg1Line8} of Algorithm \ref{th:ALg1} and using that $\gamma = p$, one can obtain
    \begin{equation*}
        \mathbb{E}\left[ \mathbb{E}_k[\|\omega^{k+1}\|^2 - \gamma\|x^{k+1}\|^2 - (1 - \gamma)\|\omega^k\|^2 ]\right] = 0.
    \end{equation*}
Using two properties above, we reach for any compact set $\mathcal{C}\subseteq X$
\begin{align*}
\mathbb{E}\left[\max _{x \in \mathcal{C}} \sum_{k=0}^{K-1} e_1(k, x)\right]
=&2 \mathbb{E}\left[\max _{x \in \mathcal{C}} \sum_{k=0}^{K-1}\left\langle u^{k+1}, x\right\rangle\right] +  \mathbb{E}\big[- \gamma \|x^{k+1}\|^2
\\&-(1-\gamma)\|w^k\|^2+\|w^k\|^2 \big]
\\
=&2 \mathbb{E}\left[\max _{x \in \mathcal{C}} \sum_{k=0}^{K-1}\left\langle u^{k+1}, x\right\rangle\right].
\end{align*}
With $\mathcal{F}_k = \sigma(\xi_0, \dots, \xi_k, x^k)$ we have that $\mathbb{E}[u^{k+1}|\mathcal{F}_k]=0$. It means that we can apply Lemma \ref{lem:1}. Thus, we get
\begin{align}
\mathbb{E}\left[\max _{x \in \mathcal{C}} \sum_{k=0}^{K-1} e_1(k, x)\right]
 \leq & 2\max _{x \in \mathcal{C}}\left\|x_0-x\right\|^2+ \frac{1}{2}\sum_{k=0}^{K-1} \mathbb{E}\left\|u^{k+1}\right\|^2.
 \label{lem4:beforeDisp}
\end{align}
We estimate $\|u_{k+1}\|^2$ using the fact that $\mathbb{E}\|X-\mathbb{E}X\|^2 = \mathbb{E}\|X\|^2-\|\mathbb{E}X\|^2$ and line \ref{Alg1Line8} of Algorithm \ref{th:ALg1}:
\begin{align*}
\mathbb{E}\left\|u_{k+1}\right\|^2 
& =
\mathbb{E}\left[\mathbb{E}_{k}\left\|u_{k+1}\right\|^2\right]
=\mathbb{E}\left[\mathbb{E}_{k}\left\|\mathbb{E}_{k}\left[\mathbf{\omega}_{k+1}\right]-\mathbf{\omega}_{k+1}\right\|^2\right] 
\\
& =
\mathbb{E}\left[\mathbb{E}_{k}\left\|\mathbf{\omega}_{k+1}\right\|^2-\left\|\mathbb{E}_{k}\left[\mathbf{\omega}_{k+1}\right]\right\|^2\right] 
\\
& =
\mathbb{E}\left[\gamma\left\|x_{k+1}\right\|^2+(1-\gamma)\left\|\mathbf{\omega}_k\right\|^2-\left\|\gamma x_{k+1}+(1-\gamma) \mathbf{\omega}_k\right\|^2\right] 
\\
& =
\gamma(1-\gamma) \mathbb{E}\left\|x_{k+1}-\omega_k\right\|^2.
\end{align*}
Applying this result to (\ref{lem4:beforeDisp}), we get
\begin{align*}
\mathbb{E}\left[\max _{x \in \mathcal{C}} \sum_{k=0}^{K-1} e_1(k, x)\right]  =2\max _{x \in \mathcal{C}}\left\|x_0-x\right\|^2+ \frac{\gamma(1-\gamma)}{2} \sum_{k=0}^{K-1} \mathbb{E}\left\|x_{k+1}-\omega_k\right\|^2.
\end{align*}
\end{proof} 
\textit{Proof of Theorem \ref{th:ALg1}}.
We start from  
\begin{align*}
\|x^{k+1}-x\|^2  = & \|x^k-x\|^2+2\langle x^{k+1}-x^k, x^{k+1}-x \rangle - \|x^{k+1}-x^k\|^2
\\
= & \|x^k-x\|^2+2\gamma\langle 
w^k-x^k, x^{k+1}-x \rangle
- 2\eta \langle\Delta^k, x^{k+1}-x\rangle
\\
&-\|x^{k+1}-x^k\|^2-2[\langle x^k+\gamma (w^k-x^k)-\eta \Delta^k-x^{k+1}, x^{k+1}-x\rangle].
\end{align*}
From line \ref{Alg1Line7} of Algorithm \ref{alg:sarah} and according to the property (\ref{prox_property}) of proximal operator, it follows, that 
\begin{align*}
    x^{k}+\gamma(w^{k}-x^{k})-\eta\Delta^{k}-x^{k+1}\in\partial(\eta g)(x^{k+1}).
\end{align*}
From convexity of $g(\cdot)$, we obtain
\begin{align*}
      \|x^{k+1}-x\|^2\leq &\|x^k-x\|^2+2\gamma\langle 
w^k-x^k, x^{k+1}-x \rangle-2\eta \langle\Delta^k, x^{k+1}-x\rangle
\\ &
-\|x^{k+1}-x^k\|^2
+ 2\eta g(x) - 2\eta g(x^{k+1}). 
\end{align*}
Using $2\gamma \langle w^k - x^k, x^{k+1} - x\rangle = 2\gamma \langle w^k - x, x^{k+1} - x\rangle - 2\gamma \langle x^k - x, x^{k+1} - x\rangle$ and the following property of scalar product: $2\langle a, b \rangle = \|a+b\|^2 - \|a\|^2- \|b\|^2$, we get
\begin{align*}
\|x^{k+1}-x\|^2 
\leq & 
\|x^k-x\|^2+\gamma\left( \|w^k-x\|^2+\|x^{k+1}-x\|^2-\|x^{k+1}-w^k\|^2\right)
\\
&
-
2\eta\langle \Delta^k, x^{k+1}-x\rangle
-\gamma\|x^{k+1}-x\|^2-\gamma \|x^k-x\|^2
\\
&
+
\gamma \|x^{k+1}-x^k\|^2
-
\|x^{k+1}-x^k\|^2 + 2 \eta g(x) - 2 \eta g(x^{k+1})
\\
= 
&
\|x^k-x\|^2+\gamma \|w^k-x\|^2 - \gamma\|x^k-x\|^2-\gamma \|x^{k+1}-w^k\|^2
\\
&
-
2\eta \langle \Delta^k,x^{k+1}-x \rangle - (1-\gamma)\|x^{k+1}-x^k\|^2 + 2\eta g(x) - 2\eta g(x^{k+1}).
\end{align*} 
Applying the properties of $\mathbb{E}_k[\Delta^k]$ specified in Lemma \ref{lem:3}, we obtain
\begin{align*}
\|x^{k+1}-x\|^2
\leq & \|x^k-x\|^2
+\gamma \left\|w^k-x\right\|^2
-\gamma \|x^k-x\|^2-\gamma \|w^k-x^{k+1}\|^2 
\\ &
- 2 \eta \langle\mathbb{E}_k\left[\Delta^k\right], x^{k+1}-x\rangle
-(1-\gamma) \|x^{k+1}-x^k\|^2
\\ & 
+2 \eta \langle\mathbb{E}_k\left[\Delta^k\right]-\Delta^k, x^{k+1}-x^k\rangle
+2 \eta \langle\mathbb{E}_k\left[\Delta^k\right]-\Delta^k, x^k-x\rangle
\\ & + 2\eta g(x) - 2\eta g(x^{k+1}) 
\\ 
= & \left\|x^k-x\right\|^2
+\gamma \left\|w^k-x\right\|^2-\gamma \left\|x^k-x\right\|^2-\gamma \left\|w^k-x^{k+1}\right\|^2 \\ &-2 \eta \left\langle F\left(x^k\right)+F\left(x^k\right)-F\left(x^{k-1}\right), x^{k+1}-x\right\rangle
\\&-(1-\gamma) \left\|x^{k+1}-x^k\right\|^2 +2 \eta \left\langle\mathbb{E}_k\left[\Delta^k\right]-\Delta^k, x^{k+1}-x^k\right\rangle
\\& +2 \eta \left\langle\mathbb{E}_k\left[\Delta^k\right]-\Delta^k, x^k-x\right\rangle 
\\&
+ 
2\eta g(x) - 2\eta g(x^{k+1})
\\
= &
\left\|x^k-x\right\|^2
+\gamma \left\|w^k-x\right\|^2-\gamma \left\|x^k-x\right\|^2-\gamma \left\|w^k-x^{k+1}\right\|^2 
\\&
-
2\eta \langle F(x^k)-F(x^{k+1})+F(x^k)-F(x^{k-1}), x^{k+1}-x \rangle
\\&
-2\eta \langle F(x^{k+1}), x^{k+1}-x \rangle
\\&
-
(1-\gamma) \left\|x^{k+1}-x^k\right\|^2 
+2 \eta \left\langle\mathbb{E}_k\left[\Delta^k\right]
-\Delta^k, x^{k+1}-x^k\right\rangle
\\&
+2 \eta \left\langle\mathbb{E}_k\left[\Delta^k\right]
-\Delta^k, x^k-x\right\rangle
+
2\eta g(x) - 2\eta g(x^{k+1})
.\end{align*}
By a simple rearrangements, we obtain 
\begin{align*}
2\eta (g(x^{k+1}) -g(x))
&+2\eta \langle F(x^{k+1}), x^{k+1}-x \rangle
\\
\leq &
\left\|x^k-x\right\|^2
+\gamma \left\|w^k-x\right\|^2-\gamma \left\|x^k-x\right\|^2
\\&-\gamma \left\|w^k-x^{k+1}\right\|^2- \|x^{k+1}-x\|^2
\\&
-
2\eta \langle F(x^k)-F(x^{k+1})+F(x^k)-F(x^{k-1}), x^{k+1}-x \rangle
\\&-(1-\gamma) \left\|x^{k+1}-x^k\right\|^2 
+
2 \eta \left\langle\mathbb{E}_k\left[\Delta^k\right]
-\Delta^k, x^{k+1}-x^k\right\rangle
\\&
+2 \eta \left\langle\mathbb{E}_k\left[\Delta^k\right]
-\Delta^k, x^k-x\right\rangle 
\\
= &
(1 - \gamma)\left\|x^k-x\right\|^2 + \left\|w^k-x\right\|^2
\\
&-(1-\gamma)\left\|x^{k+1}-x\right\|^2 -  \left\|w^{k+1}-x\right\|^2
\\
&+ \left\|w^{k+1}-x\right\|^2 -\gamma\left\|x^{k+1}-x\right\|^2
\\
&- (1 - \gamma) \left\|w^k-x\right\|^2-\gamma \left\|w^k-x^{k+1}\right\|^2
\\
&-
2\eta \langle F(x^k)-F(x^{k+1}), x^{k+1}-x \rangle
\\
&+2\eta \langle F(x^{k-1})-F(x^k), x^{k+1}-x \rangle
\\
&-
(1-\gamma) \left\|x^{k+1}-x^k\right\|^2 
+2 \eta \left\langle\mathbb{E}_k\left[\Delta^k\right]
-\Delta^k, x^{k+1}-x^k\right\rangle
\\&
+2 \eta \left\langle\mathbb{E}_k\left[\Delta^k\right]
-\Delta^k, x^k-x\right\rangle
\\
= &
(1 - \gamma)\left\|x^k-x\right\|^2 + \left\|w^k-x\right\|^2
\\
&-(1-\gamma)\left\|x^{k+1}-x\right\|^2 -  \left\|w^{k+1}-x\right\|^2
\\
&+ \left\|w^{k+1}-x\right\|^2 -\gamma\left\|x^{k+1}-x\right\|^2
\\
&- (1 - \gamma) \left\|w^k-x\right\|^2-\gamma \left\|w^k-x^{k+1}\right\|^2
\\&
-
2\eta \langle F(x^k)-F(x^{k+1}), x^{k+1}-x \rangle
\\&
+2\eta \langle F(x^{k-1})-F(x^k), x^k-x \rangle
\\&
+2\eta \langle F(x^{k-1})-F(x^k), x^{k+1}-x^k \rangle
\\&
-
(1-\gamma) \left\|x^{k+1}-x^k\right\|^2
+2 \eta \left\langle\mathbb{E}_k\left[\Delta^k\right]
-\Delta^k, x^{k+1}-x^k\right\rangle
\\&
+2 \eta \left\langle\mathbb{E}_k\left[\Delta^k\right]
-\Delta^k, x^k-x\right\rangle.
\end{align*}
After taking sum and then averaging, one can get
\begin{align}
\label{eq:temp30457}
2\eta \cdot \frac{1}{K}\sum\limits_{k=0}^{K-1}
&\left[\langle F(x^{k+1}), x^{k+1}-x \rangle + g(x^{k+1}) - g(x) \right]
\notag\\
\leq & 
\frac{1}{K}\sum\limits_{k=0}^{K-1} \left[(1-\gamma)\left\|x^k-x\right\|^2 +  \left\|w^k-x\right\|^2\right]
\notag\\&
-\frac{1}{K}\sum\limits_{k=0}^{K-1} \left[(1-\gamma)\left\|x^{k+1}-x\right\|^2 +  \left\|w^{k+1}-x\right\|^2\right]
\notag\\&
-
2\eta \cdot \frac{1}{K}\sum\limits_{k=0}^{K-1} \langle F(x^k)-F(x^{k+1}), x^{k+1}-x \rangle
\notag\\&
+
2\eta \cdot \frac{1}{K}\sum\limits_{k=0}^{K-1} \langle F(x^{k-1})-F(x^k), x^k-x \rangle
\notag\\& 
+\frac{1}{K} \sum\limits_{k=0}^{K-1} \left[ \left\|w^{k+1}-x\right\|^2 - \gamma  \left\|x^{k+1}-x\right\|^2 - (1 - \gamma)\left\|w^{k}-x\right\|^2\right]
\notag\\&
+2 \eta \cdot \frac{1}{K} \sum\limits_{k=0}^{K-1}\left\langle\mathbb{E}_k\left[\Delta^k\right]
-\Delta^k, x^k-x\right\rangle
\notag\\& 
-\frac{\gamma}{K}\sum\limits_{k=0}^{K-1}  \left\|w^k-x^{k+1}\right\|^2 - \frac{1 - \gamma}{K}\sum\limits_{k=0}^{K-1} \left\|x^{k+1}-x^k\right\|^2
\notag\\&
+2\eta \cdot \frac{1}{K}\sum\limits_{k=0}^{K-1} \langle F(x^{k-1})-F(x^k), x^{k+1}-x^k \rangle
\notag\\&
+2 \eta \cdot \frac{1}{K}\sum\limits_{k=0}^{K-1}\left\langle\mathbb{E}_k\left[\Delta^k\right]
-\Delta^k, x^{k+1}-x^k\right\rangle
\notag\\ 
= & 
\frac{2-\gamma}{K}\left\|x^0-x\right\|^2 - \frac{1 - \gamma}{K}\left\|x^K-x\right\|^2 - \frac{1}{K} \left\|w^K-x\right\|^2
\notag\\&
-
2 \eta \cdot \frac{1}{K} \langle F(x^{K-1})-F(x^{K}), x^{K}-x \rangle
\notag\\& 
+\frac{1}{K}\sum\limits_{k=0}^{K-1} \left[ \left\|w^{k+1}-x\right\|^2 - \gamma  \left\|x^{k+1}-x\right\|^2 - (1 - \gamma)\left\|w^{k}-x\right\|^2\right]
\notag\\&
+2 \eta \cdot\frac{1}{K}\sum\limits_{k=0}^{K-1}\left\langle\mathbb{E}_k\left[\Delta^k\right]
-\Delta^k, x^k-x\right\rangle
\notag\\& 
-\frac{\gamma}{K}\sum\limits_{k=0}^{K-1}  \left\|w^k-x^{k+1}\right\|^2 - \frac{1 - \gamma}{K}\sum\limits_{k=0}^{K-1} \left\|x^{k+1}-x^k\right\|^2
\notag\\&
+2\eta \cdot\frac{1}{K}\sum\limits_{k=0}^{K-1} \langle F(x^{k-1})-F(x^k), x^{k+1}-x^k \rangle
\notag\\&
+2 \eta \cdot \frac{1}{K}\sum\limits_{k=0}^{K-1}\left\langle\mathbb{E}_k\left[\Delta^k\right]
-\Delta^k, x^{k+1}-x^k\right\rangle
.\end{align}
Here we also used the initialization of Algorithm \ref{alg:sarah} with $w^0 = x^{-1} = x^0$. Applying Young's inequality, 
using the $L$-Lipshetzness of $F$, and taking into account the definition of $\eta \leq \frac{1}{8L}$ from conditions of the theorem for any $k$, one can obtain
\begin{align}
\label{eq:temp30456}
-
2 \eta \langle F(x^{K-1})-F(x^{K}), x^{K}-x \rangle
\leq & 
2\eta^2 \left\| F(x^{K-1})-F(x^{K}) \right\|^2 +
\frac{1}{2}\left\| x^{K}-x \right\|^2
\notag\\
\leq & 
2\eta^2 L^2 \left\| x^{K-1} -x^{K} \right\|^2 +
\frac{1}{2}\left\| x^{K}-x \right\|^2
\notag\\
\leq & 
\frac{1}{32} \left\| x^{K-1} -x^{K} \right\|^2 +
\frac{1}{2}\left\| x^{K}-x \right\|^2
.\end{align}
Combining \eqref{eq:temp30457} and \eqref{eq:temp30456}, we get
\begin{align*}
2\eta \cdot \frac{1}{K}
\sum\limits_{k=0}^{K-1} &\big[ \langle F(x^{k+1}), x^{k+1}-x \rangle + g(x^{k+1}) - g(x) \big]
\notag\\
\leq & 
\frac{2 - \gamma}{K}\left\|x^0-x\right\|^2 - \frac{1}{K}\left(\frac{1}{2} - \gamma \right)\left\|x^K-x\right\|^2 - \frac{1}{K}\left\|w^K-x\right\|^2
\notag
\\& 
+\frac{1}{K}\sum\limits_{k=0}^{K-1} \left[ \left\|w^{k+1}-x\right\|^2 - \gamma  \left\|x^{k+1}-x\right\|^2 - (1 - \gamma)\left\|w^{k}-x\right\|^2\right]
\notag\\&
+\frac{1}{32K} \left\| x^{K-1} -x^{K} \right\|^2+2 \eta \cdot \frac{1}{K}\sum\limits_{k=0}^{K-1}\left\langle\mathbb{E}_k\left[\Delta^k\right]
-\Delta^k, x^k-x\right\rangle
\notag\\& 
-\frac{\gamma}{K}\sum\limits_{k=0}^{K-1}  \left\|w^k-x^{k+1}\right\|^2 - \frac{1 - \gamma}{K}\sum\limits_{k=0}^{K-1} \left\|x^{k+1}-x^k\right\|^2
\notag\\&
+2\eta \cdot \frac{1}{K}\sum\limits_{k=0}^{K-1} \langle F(x^{k-1})-F(x^k), x^{k+1}-x^k \rangle
\notag\\&
+2 \eta \cdot \frac{1}{K}\sum\limits_{k=0}^{K-1}\left\langle\mathbb{E}_k\left[\Delta^k\right]
-\Delta^k, x^{k+1}-x^k\right\rangle
\\
\leq & 
\frac{2 - \gamma}{K}\left\|x^0-x\right\|^2
\notag
\\&
+\frac{1}{K}\sum\limits_{k=0}^{K-1} \left[ \left\|w^{k+1}-x\right\|^2 - \gamma  \left\|x^{k+1}-x\right\|^2 - (1 - \gamma)\left\|w^{k}-x\right\|^2\right]
\notag\\&
+2 \eta \cdot \frac{1}{K}\sum\limits_{k=0}^{K-1}\left\langle\mathbb{E}_k\left[\Delta^k\right]
-\Delta^k, x^k-x\right\rangle -\frac{\gamma}{K}\sum\limits_{k=0}^{K-1}  \left\|w^k-x^{k+1}\right\|^2
\notag\\& 
- \frac{1 - \gamma}{K}\sum\limits_{k=0}^{K-1} \left\|x^{k+1}-x^k\right\|^2 +\frac{1}{32K} \left\| x^{K-1} -x^{K} \right\|^2 
\notag\\&
+2\eta \cdot \frac{1}{K}\sum\limits_{k=0}^{K-1} \langle F(x^{k-1})-F(x^k), x^{k+1}-x^k \rangle
\notag\\&
+2 \eta \cdot \frac{1}{K}\sum\limits_{k=0}^{K-1}\left\langle\mathbb{E}_k\left[\Delta^k\right]
-\Delta^k, x^{k+1}-x^k\right\rangle
.\end{align*}
Next, we use monotonicity of $F$, apply Jensen's inequality for the convex function $g$ and obtain
\begin{align*}
2\eta \big[ \langle F(x), &\frac{1}{K}
\sum\limits_{k=0}^{K-1} x^{k+1}-x \rangle  + g\left(\frac{1}{K}
\sum\limits_{k=0}^{K-1} x^{k+1}\right) - g(x) \big]
\\
\leq & 
\frac{2 - \gamma}{K}\left\|x^0-x\right\|^2
\notag
\\&
+\frac{1}{K}\sum\limits_{k=0}^{K-1} \left[ \left\|w^{k+1}-x\right\|^2 - \gamma  \left\|x^{k+1}-x\right\|^2 - (1 - \gamma)\left\|w^{k}-x\right\|^2\right]
\notag\\&
+2 \eta \cdot \frac{1}{K}\sum\limits_{k=0}^{K-1}\left\langle\mathbb{E}_k\left[\Delta^k\right]
-\Delta^k, x^k-x\right\rangle-\frac{\gamma}{K}\sum\limits_{k=0}^{K-1}  \left\|w^k-x^{k+1}\right\|^2 
\notag\\& 
- \frac{1 - \gamma}{K}\sum\limits_{k=0}^{K-1} \left\|x^{k+1}-x^k\right\|^2 +\frac{1}{32K} \left\| x^{K-1} -x^{K} \right\|^2 
\notag\\&
+2\eta \cdot \frac{1}{K}\sum\limits_{k=0}^{K-1} \langle F(x^{k-1})-F(x^k), x^{k+1}-x^k \rangle
\notag\\&
+2 \eta \cdot \frac{1}{K}\sum\limits_{k=0}^{K-1}\left\langle\mathbb{E}_k\left[\Delta^k\right]
-\Delta^k, x^{k+1}-x^k\right\rangle
.\end{align*}
Using new notation $\bar x^K = \frac{1}{K}
\sum\limits_{k=0}^{K-1} x^{k+1}$ and taking maximum on $\mathcal{C}$,  we achieve
\begin{align*}
2\eta &\text{Gap} (\bar x^K) \leq 
\max\limits_{x\in \mathcal{C}} \Bigg\{ \frac{2 - \gamma}{K}\left\|x^0-x\right\|^2 
\\& 
+\frac{1}{K}\sum\limits_{k=0}^{K-1} \left[ \left\|w^{k+1}-x\right\|^2 - \gamma  \left\|x^{k+1}-x\right\|^2 - (1 - \gamma)\left\|w^{k}-x\right\|^2\right]
\\&
+2 \eta \cdot \frac{1}{K}\sum\limits_{k=0}^{K-1}\left\langle\mathbb{E}_k\left[\Delta^k\right]
-\Delta^k, x^k-x\right\rangle \Bigg\} + \frac{1}{32 K} \left\| x^{K-1} -x^{K} \right\|^2
\\& 
-\frac{\gamma}{K}\sum\limits_{k=0}^{K-1}  \left\|w^k-x^{k+1}\right\|^2 -\frac{1 - \gamma}{K}\sum\limits_{k=0}^{K-1} \left\|x^{k+1}-x^k\right\|^2
\\&
+2\eta \cdot \frac{1}{K}\sum\limits_{k=0}^{K-1} \langle F(x^{k-1})-F(x^k), x^{k+1}-x^k \rangle
\\&
+2 \eta \cdot \frac{1}{K}\sum\limits_{k=0}^{K-1}\left\langle\mathbb{E}_k\left[\Delta^k\right]
-\Delta^k, x^{k+1}-x^k\right\rangle
\\ \leq &
\max\limits_{x\in \mathcal{C}} \left\{\frac{2-\gamma}{K} \left\|x^0-x\right\|^2  \right\}
\\&+\max\limits_{x\in \mathcal{C}} \Bigg\{\frac{1}{K}\sum\limits_{k=0}^{K-1} \left[ \left\|w^{k+1}-x\right\|^2
- \gamma  \left\|x^{k+1}-x\right\|^2- (1 - \gamma)\left\|w^{k}-x\right\|^2\right] \Bigg\}
\\&
+2 \eta \max\limits_{x\in \mathcal{C}} \left\{\frac{1}{K} \sum\limits_{k=0}^{K-1}\left\langle\mathbb{E}_k\left[\Delta^k\right]
-\Delta^k, x^k-x\right\rangle \right\} + \frac{1}{32 K} \left\| x^{K-1} -x^{K} \right\|^2
\\& 
-\frac{\gamma}{K}\sum\limits_{k=0}^{K-1}  \left\|w^k-x^{k+1}\right\|^2 - \frac{1 -\gamma}{K}\sum\limits_{k=0}^{K-1} \left\|x^{k+1}-x^k\right\|^2
\\&
+2\eta \cdot \frac{1}{K}\sum\limits_{k=0}^{K-1} \langle F(x^{k-1})-F(x^k), x^{k+1}-x^k \rangle
\\&
+2 \eta \cdot \frac{1}{K}\sum\limits_{k=0}^{K-1}\left\langle\mathbb{E}_k\left[\Delta^k\right]
-\Delta^k, x^{k+1}-x^k\right\rangle
.\end{align*}
Here we also used that maximum of the sum not greater than the sum of the maximums.
After that we take the an expectation and get
\begin{align*}
2\eta 
\mathbb{E}\left[\text{Gap}  (\bar x^K) \right]\notag
\leq & \mathbb{E}\Bigg[\max\limits_{x\in \mathcal{C}} \Bigg\{\frac{2 - \gamma}{K}\left\|x^0-x\right\|^2
\Bigg\} \Bigg] \notag
\\&
+\mathbb{E}\Bigg[\max\limits_{x\in \mathcal{C}} \Bigg\{\frac{1}{K}\sum\limits_{k=0}^{K-1} \Big[ \left\|w^{k+1}-x\right\|^2 - \gamma  \left\|x^{k+1}-x\right\|^2
\\&- (1 - \gamma)\left\|w^{k}-x\right\|^2\Big] \Bigg\}\Bigg]+ \frac{1}{32 K} \mathbb{E}\left[\left\| x^{K-1} -x^{K} \right\|^2\right]\notag
\\&
+2 \eta \mathbb{E}\Bigg[\max\limits_{x\in \mathcal{C}} \left\{\frac{1}{K}\sum\limits_{k=0}^{K-1}\left\langle\mathbb{E}_k\left[\Delta^k\right]
-\Delta^k, x^k-x\right\rangle \right\}\Bigg] \notag
\\& 
-\mathbb{E}\Bigg[\frac{\gamma}{K} \sum\limits_{k=0}^{K-1}  \left\|w^k-x^{k+1}\right\|^2 
+ \frac{1 - \gamma}{K} \sum\limits_{k=0}^{K-1} \left\|x^{k+1}-x^k\right\|^2\Bigg]\notag
\\&
+2\eta \mathbb{E}\Bigg[\frac{1}{K} \sum\limits_{k=0}^{K-1} \langle F(x^{k-1})-F(x^k), x^{k+1}-x^k \rangle\Bigg]\notag
\\&
+2 \eta \mathbb{E}\Bigg[\frac{1}{K} \sum\limits_{k=0}^{K-1}\left\langle\mathbb{E}_k\left[\Delta^k\right]
-\Delta^k, x^{k+1}-x^k\right\rangle\Bigg]
.
\end{align*}
With Lemma \ref{lem:4} for the second line of the previous estimate and Lemma \ref{lem:1} for the third line, we get
\begin{align}
2\eta 
\mathbb{E}\left[\text{Gap} (\bar x^K) \right]\notag
\leq & \mathbb{E}\Bigg[\max\limits_{x\in \mathcal{C}} \left\{\frac{2 - \gamma}{K}\left\|x^0-x\right\|^2
\right\} \Bigg] \notag
\\&
+\max\limits_{x \in \mathcal{C}} \left\{\frac{2}{K}\|x - x^0\|^2 \right\} + \frac{\gamma(1 - \gamma)}{2K}\sum\limits_{k = 0}^{K - 1}\mathbb{E}\left[\|x^{k+1} - \omega^k\|^2\right] \notag
\\&
+\max\limits_{x \in \mathcal{C}} \left\{\frac{1}{K}\|x - x^0\|^2 \right\} + \frac{\eta^2}{K}\sum\limits_{k = 0}^{K - 1}\mathbb{E}\left[\|\mathbb{E}_k\left[\Delta^k\right] -\Delta^k\|^2\right] \notag
\\& 
-\mathbb{E}\Bigg[\frac{\gamma}{K} \sum\limits_{k=0}^{K-1}  \left\|w^k-x^{k+1}\right\|^2 
+ \frac{1 - \gamma}{K} \sum\limits_{k=0}^{K-1} \left\|x^{k+1}-x^k\right\|^2\Bigg] \notag
\\&+ \frac{1}{32 K} \mathbb{E}\left[\left\| x^{K-1} -x^{K} \right\|^2\right]\notag
\\&
+2\eta \mathbb{E}\Bigg[\frac{1}{K} \sum\limits_{k=0}^{K-1} \langle F(x^{k-1})-F(x^k), x^{k+1}-x^k \rangle\Bigg]\notag
\\&
+2 \eta \mathbb{E}\Bigg[\frac{1}{K} \sum\limits_{k=0}^{K-1}\left\langle\mathbb{E}_k\left[\Delta^k\right]
-\Delta^k, x^{k+1}-x^k\right\rangle\Bigg] \notag
\\
\leq & \frac{4}{K} \mathbb{E}\Bigg[\max\limits_{x\in \mathcal{C}} \left\{\left\|x^0-x\right\|^2
\right\} \Bigg] 
+ \frac{\eta^2}{K}\sum\limits_{k = 0}^{K - 1}\mathbb{E}\left[\|\mathbb{E}_k\left[\Delta^k\right] -\Delta^k\|^2\right]\notag
\\& 
-\mathbb{E}\Bigg[\frac{\gamma}{2K} \sum\limits_{k=0}^{K-1}  \left\|w^k-x^{k+1}\right\|^2 
+ \frac{1 - \gamma}{K} \sum\limits_{k=0}^{K-1} \left\|x^{k+1}-x^k\right\|^2\Bigg]\notag
\\&+ \frac{1}{32 K} \mathbb{E}\left[\left\| x^{K-1} -x^{K} \right\|^2\right]\notag
\\&
+2\eta \mathbb{E}\Bigg[\frac{1}{K} \sum\limits_{k=0}^{K-1} \langle F(x^{k-1})-F(x^k), x^{k+1}-x^k \rangle\Bigg]\notag
\\&
+2 \eta \mathbb{E}\Bigg[\frac{1}{K} \sum\limits_{k=0}^{K-1}\left\langle\mathbb{E}_k\left[\Delta^k\right]
-\Delta^k, x^{k+1}-x^k\right\rangle\Bigg].
\label{then}
\end{align}
According to Young's inequality,
\begin{align}
    \mathbb{E}\big[2 \eta \langle  \mathbb{E}_k\left[\Delta^k\right] -\Delta^k , x^{k+1}-x^k \rangle\big]
    \leq
    4\eta^2\mathbb{E}\left[\|\mathbb{E}_k\left[\Delta^k\right]-\Delta^k\|^2\right]
    +
    \frac{1}{4}\mathbb{E}\left[\|x^{k+1}-x^k\|^2\right],
    \label{tech25}
\end{align}
and
\begin{align}
    \mathbb{E}\big[2 \eta \langle F(x^{k-1})-F(x^k), x^{k+1}&-x^k \rangle \big]
    \\\leq&
    4\eta^2\mathbb{E}\left[\|F(x^{k-1})-F(x^k)\|^2\right]
    +
    \frac{1}{4}\mathbb{E}\left[\|x^{k+1}-x^k\|^2\right].
    \label{tech35}
\end{align}
Combining \eqref{tech25}, \eqref{tech35} with \eqref{then}, we obtain
\begin{align*}
2\eta 
\mathbb{E}\left[\text{Gap} (\bar x^K) \right]\notag
\leq & \frac{4}{K} \mathbb{E}\Bigg[\max\limits_{x\in \mathcal{C}} \left\{\left\|x^0-x\right\|^2
\right\} \Bigg] 
+ \frac{\eta^2}{K}\sum\limits_{k = 0}^{K - 1}\mathbb{E}\left[\|\mathbb{E}_k\left[\Delta^k\right] -\Delta^k\|^2\right]\notag
\\& 
-\mathbb{E}\Bigg[\frac{\gamma}{2K} \sum\limits_{k=0}^{K-1}  \left\|w^k-x^{k+1}\right\|^2 
+ \frac{1 - \gamma}{K} \sum\limits_{k=0}^{K-1} \left\|x^{k+1}-x^k\right\|^2\Bigg] 
\\&+ \frac{1}{32 K} \mathbb{E}\left[\left\| x^{K-1} -x^{K} \right\|^2\right]\notag
+\frac{4\eta^2}{K}\sum\limits_{k=0}^{K-1} \mathbb{E}\left[\|F(x^{k-1})-F(x^k)\|^2\right] 
\\&+ \frac{1}{4K}\sum\limits_{k=0}^{K-1} \mathbb{E}\left[\|x^{k+1}-x^k\|^2\right]\notag+\frac{4\eta^2}{K}\mathbb{E}\sum\limits_{k=0}^{K-1} \left[\|\mathbb{E}_k\left[\Delta^k\right]-\Delta^k\|^2\right] 
\\&
+\frac{1}{4K} \sum\limits_{k=0}^{K-1} \mathbb{E}\left[\|x^{k+1}-x^k\|^2\right]
    \\
\leq & \frac{4}{K} \mathbb{E}\Bigg[\max\limits_{x\in \mathcal{C}} \left\{\left\|x^0-x\right\|^2
\right\} \Bigg] 
+ \frac{5\eta^2}{K}\sum\limits_{k = 0}^{K - 1}\mathbb{E}\left[\|\mathbb{E}_k\left[\Delta^k\right] -\Delta^k\|^2\right]\notag
\\& 
-\mathbb{E}\Bigg[\frac{\gamma}{2K} \sum\limits_{k=0}^{K-1}  \left\|w^k-x^{k+1}\right\|^2 
+ \frac{1/2 - \gamma}{K} \sum\limits_{k=0}^{K-1} \left\|x^{k+1}-x^k\right\|^2\Bigg]
\\&+ \frac{1}{32 K} \mathbb{E}\left[\left\| x^{K-1} -x^{K} \right\|^2\right]
+\frac{4\eta^2}{K}\sum\limits_{k=0}^{K-1} \mathbb{E}\left[\|F(x^{k-1})-F(x^k)\|^2\right].
\end{align*}
$L$ - Lipschitzness of $F$ (Assumption \ref{ass}) and the choice of $\gamma \leq \frac{1}{L}$ give
\begin{align*}
2\eta 
\mathbb{E}\left[\text{Gap} (\bar x^K) \right]\notag
\leq & \frac{4}{K} \mathbb{E}\Bigg[\max\limits_{x\in \mathcal{C}} \left\{\left\|x^0-x\right\|^2
\right\} \Bigg] 
+ \frac{5\eta^2}{K}\sum\limits_{k = 0}^{K - 1}\mathbb{E}\left[\|\mathbb{E}_k\left[\Delta^k\right] -\Delta^k\|^2\right]\notag
\\& 
-\mathbb{E}\Bigg[\frac{\gamma}{2K} \sum\limits_{k=0}^{K-1}  \left\|w^k-x^{k+1}\right\|^2 
+ \frac{1/2 - \gamma}{K} \sum\limits_{k=0}^{K-1} \left\|x^{k+1}-x^k\right\|^2\Bigg] 
\\&+ \frac{1}{32 K} \mathbb{E}\left[\left\| x^{K-1} -x^{K} \right\|^2\right]+\frac{4\eta^2 L^2 }{K}\sum\limits_{k=0}^{K-1} \mathbb{E}\left[\|x^{k-1}-x^k\|^2\right]
\\
\leq & \frac{4}{K} \mathbb{E}\Bigg[\max\limits_{x\in \mathcal{C}} \left\{\left\|x^0-x\right\|^2
\right\} \Bigg] 
+ \frac{5\eta^2}{K}\sum\limits_{k = 0}^{K - 1}\mathbb{E}\left[\|\mathbb{E}_k\left[\Delta^k\right] -\Delta^k\|^2\right]\notag
\\& 
-\mathbb{E}\Bigg[\frac{\gamma}{2K} \sum\limits_{k=0}^{K-1}  \left\|w^k-x^{k+1}\right\|^2 
+ \frac{1/2 - \gamma}{K} \sum\limits_{k=0}^{K-1} \left\|x^{k+1}-x^k\right\|^2\Bigg]
\\&+ \frac{1}{32 K} \mathbb{E}\left[\left\| x^{K-1} -x^{K} \right\|^2\right]+\frac{1}{4K}\sum\limits_{k=0}^{K-1} \mathbb{E}\left[\|x^{k-1}-x^k\|^2\right]
\\
\leq & \frac{4}{K} \mathbb{E}\Bigg[\max\limits_{x\in \mathcal{C}} \left\{\left\|x^0-x\right\|^2
\right\} \Bigg] 
+ \frac{5\eta^2}{K}\sum\limits_{k = 0}^{K - 1}\mathbb{E}\left[\|\mathbb{E}_k\left[\Delta^k\right] -\Delta^k\|^2\right]\notag
\\& 
-\mathbb{E}\Bigg[\frac{\gamma}{2K} \sum\limits_{k=0}^{K-1}  \left\|w^k-x^{k+1}\right\|^2 
+ \frac{1/2 - \gamma}{K} \sum\limits_{k=0}^{K-1} \left\|x^{k+1}-x^k\right\|^2\Bigg]
\\&
+\frac{1}{4K}\sum\limits_{k=0}^{K-1} \mathbb{E}\left[\|x^{k+1}-x^k\|^2\right]
\\
= & \frac{4}{K} \mathbb{E}\Bigg[\max\limits_{x\in \mathcal{C}} \left\{\left\|x^0-x\right\|^2
\right\} \Bigg] 
+ \frac{5\eta^2}{K}\sum\limits_{k = 0}^{K - 1}\mathbb{E}\left[\|\mathbb{E}_k\left[\Delta^k\right] -\Delta^k\|^2\right]\notag
\\& 
-\mathbb{E}\Bigg[\frac{\gamma}{2K} \sum\limits_{k=0}^{K-1}  \left\|w^k-x^{k+1}\right\|^2 
+ \frac{1/4 - \gamma}{K} \sum\limits_{k=0}^{K-1} \left\|x^{k+1}-x^k\right\|^2\Bigg].
\end{align*}
Here we also used the initialization of Algorithm \ref{alg:sarah} with $x^{-1}= x^0$. Applying Lemma \ref{lem:3}, we obtain
\begin{align*}
2\eta 
\mathbb{E}\left[\text{Gap} (\bar x^K) \right]\notag
\leq & \frac{4}{K} \mathbb{E}\Bigg[\max\limits_{x\in \mathcal{C}} \left\{\left\|x^0-x\right\|^2
\right\} \Bigg]
\\&
+ \frac{10\eta^2 {\overline{{L}}}^{2} }{b K}\sum\limits_{k = 0}^{K - 1} \left(\|x^{k}-w^{k-1}\|^{2}+\left\|x^{k}-x^{k-1}\right\|^{2}\right)
\\& 
-\mathbb{E}\Bigg[\frac{\gamma}{2K} \sum\limits_{k=0}^{K-1}  \left\|w^k-x^{k+1}\right\|^2 
+ \frac{1/4 - \gamma}{K} \sum\limits_{k=0}^{K-1} \left\|x^{k+1}-x^k\right\|^2\Bigg]
\\
\leq & \frac{4}{K} \mathbb{E}\Bigg[\max\limits_{x\in \mathcal{C}} \left\{\left\|x^0-x\right\|^2
\right\} \Bigg]
\\&
+ \frac{10\eta^2 {\overline{{L}}}^{2} }{b K}\sum\limits_{k = 0}^{K - 1} \left(\|x^{k+1}-w^{k}\|^{2}+\left\|x^{k}-x^{k+1}\right\|^{2}\right)
\\& 
-\mathbb{E}\Bigg[\frac{\gamma}{2K} \sum\limits_{k=0}^{K-1}  \left\|w^k-x^{k+1}\right\|^2 
+ \frac{1/4 - \gamma}{K} \sum\limits_{k=0}^{K-1} \left\|x^{k+1}-x^k\right\|^2\Bigg]
\\
\leq & \frac{4}{K} \mathbb{E}\Bigg[\max\limits_{x\in \mathcal{C}} \left\{\left\|x^0-x\right\|^2
\right\} \Bigg]
\\& 
-\mathbb{E}\Bigg[ \left( \frac{\gamma}{2} - \frac{10\eta^2 {\overline{{L}}}^{2} }{b}\right)\frac{1}{K} \sum\limits_{k=0}^{K-1}  \left\|w^k-x^{k+1}\right\|^2 
\\&
+ \left(\frac{1}{4} - \gamma -  \frac{10\eta^2 {\overline{{L}}}^{2} }{b}\right)\frac{1}{K} \sum\limits_{k=0}^{K-1} \left\|x^{k+1}-x^k\right\|^2\Bigg].
\end{align*}
Here we again used the initialization of Algorithm \ref{alg:sarah} with $w^{-1} = x^{-1}= x^0$. The choice of $\eta \leq \frac{\sqrt{\gamma b}}{8 \bar L}$ and $ 0 < \gamma \leq \frac{1}{16}$ gives
\begin{align*}
2\eta 
\mathbb{E}\left[\text{Gap} (\bar x^K) \right]\notag
\leq & \frac{4}{K} \mathbb{E}\Bigg[\max\limits_{x\in \mathcal{C}} \left\{\left\|x^0-x\right\|^2
\right\} \Bigg]
\\& 
-\mathbb{E}\Bigg[ \left(\frac{1}{12} - \gamma \right)\frac{1}{K} \sum\limits_{k=0}^{K-1} \left\|x^{k+1}-x^k\right\|^2\Bigg]
\\
\leq & \frac{4}{K} \max\limits_{x\in \mathcal{C}} \left\{\left\|x^0-x\right\|^2 \right\}.
\end{align*}
And we have 
\begin{align*} 
\mathbb{E}\left[\text{Gap} (\bar x^K) \right]\notag
\leq & \frac{2}{\eta K} \max\limits_{x\in \mathcal{C}} \left\{\left\|x^0-x\right\|^2 \right\}.
\end{align*}
Substitution of $\eta$ from the conditions of the theorem and $\gamma = p$ finishes the proof.

\end{document}